\documentclass[a4paper,11pt]{amsart}

\usepackage{amsmath,amssymb,amsthm} 
\usepackage{graphics}
\usepackage{comment}
\usepackage{tikz}
\usepackage{float}
\usepackage{thm-restate}
\usepackage{thmtools}
\usepackage{enumerate}
\usetikzlibrary{calc}
\usetikzlibrary{patterns}
\usetikzlibrary{spy}
\usetikzlibrary{decorations.pathreplacing}

\numberwithin{equation}{section}

\newtheorem{theorem}{Theorem}[section]
\newtheorem{lemma}[theorem]{Lemma}
\newtheorem{proposition}[theorem]{Proposition}
\newtheorem{corollary}[theorem]{Corollary}

\theoremstyle{definition}
\newtheorem{definition}[theorem]{Definition}
\newtheorem{remark}[theorem]{Remark}

\theoremstyle{plain}
\newtheorem*{claim*}{Claim}

\newcommand{\R}{\mathbb{R}}

\newcommand{\Ups}{\Upsilon}
\newcommand{\E}{\mathcal{E}}
\newcommand{\vhi}{\varphi}
\newcommand{\eps}{\varepsilon}
\newcommand{\PI}{{\varPi}}
\newcommand{\om}{\omega}
\newcommand{\Om}{\Omega}
\newcommand{\lt}{\left}
\newcommand{\rt}{\right}
\newcommand{\st}{\stackrel}
\newcommand{\ds}{\displaystyle}

\newcommand{\restr} {
	\hskip2.5pt{\vrule height7pt width.5pt depth0pt}
	\hskip-.2pt\vbox{\hrule height.5pt width7pt depth0pt}
	\, }
\newcommand{\ie}{\textit{i.e.~}}

\DeclareMathOperator{\supp}{supp}
\DeclareMathOperator{\diam}{diam}

\title[An exterior optimal transport problem]{An exterior optimal transport problem}
\author[J. Candau-Tilh]{Jules Candau-Tilh}
\address{J. C-T.: Univ. Lille, CNRS, UMR 8524, Inria - Laboratoire Paul Painlev\'e, F-59000 Lille}
\email{jules.candautilh@univ-lille.fr}
\author[M. Goldman]{Michael Goldman}
\address{M.G.: CMAP, CNRS, \'Ecole polytechnique, Institut Polytechnique de Paris, 91120 Palaiseau,
	France}
\email{michael.goldman@cnrs.fr}
\author[B. Merlet]{Benoit Merlet}
\address{B.M.: Univ. Lille, CNRS, UMR 8524, Inria - Laboratoire Paul Painlev\'e, F-59000 Lille}
\email{benoit.merlet@univ-lille.fr}

\begin{document}
	\begin{abstract}
		This paper deals with a variant of the optimal transportation problem. Given $f \in L^1( \R^d, [0,1])$ and a cost function $c \in C(\R^d \times \R^d)$ of the form $c(x,y)=k(y-x)$, we minimise $ \smallint c \,d\gamma$ among transport plans $\gamma$ whose first marginal is $f$ and whose second marginal is not prescribed but constrained to be smaller than $1-f$. Denoting by $\Ups(f)$ the infimum of this problem, we then consider the maximisation problem $\sup \{\Ups(f) : \, \smallint f = m \}$ where $m > 0$ is given. We prove that maximisers exist under general assumptions on $k$, and that for $k$ radial,  increasing and coercive these maximisers are the characteristic functions of the balls of volume~$m$.\medskip
		
		\noindent
		\textbf{Keywords and phrases.} Optimal transport, dual problem, existence of maximisers.
		\medskip
		
		\noindent
		\textbf{2020 Mathematics Subject Classification.} 49Q22, 49Q20, 49J35.
	\end{abstract}
	
	\maketitle
	
	\section{Introduction}
	
	In this paper, we study the optimization problems associated with functionals which favour dispersion and are based on some Wasserstein energies. These functionals correspond to the non-local term of the energy studied in~\cite{CanGol,BCL2020,PeRo09,XiaZhou,NoToVenk}. Our main result is that for a very large class of radial costs, balls are the unique volume-constrained maximisers of these functionals. This confirms that they enter in strong competition with the perimeter for which balls are volume-constrained minimisers. \medskip
	
	We  denote by $\mathcal{M}_+(\R^d)$ the set of positive Radon measures on $\R^d$. Given a cost function $c$ and $\mu, \nu \in \mathcal{M}_+(\R^d)$, we let $\mathcal{T}_c(\mu, \nu)$ be the $c$-transport cost between $\mu$ and $\nu$ (see Section~\ref{StdOT} for the exact definition of $\mathcal{T}_c$).  Given a measurable set $E \subset \R^d$ with finite volume, we consider the optimisation problem
	\begin{equation}\label{Upsfunc}
		\Ups_{\mathrm{set}}(E) := \inf\lt\{ \mathcal{T}_c(E, F):F\subset \R^d\text{ Lebesgue measurable}, |F|=|E|, \, |F\cap E| = 0\rt\}
	\end{equation}
	where we identify $E$ with the restriction of the Lebesgue measure on $E$. Given $m >0$, we introduce the maximisation problem
	\begin{equation}\label{maxE}
		\E_{\mathrm{set}}(m) := \sup_{|E|=m} \Ups_{\mathrm{set}}(E).
	\end{equation}	
	The main goal of the article is to investigate the existence of maximisers for this problem and to characterise these latter.\\
	If we apply the direct method of the Calculus of Variations, we obtain that, up to extraction, any maximising sequence $E_n$ converges weakly  to some function $u_\infty\in L^1(\R^d, [0,1])$. However, there is no guarantee at this point that $u_\infty$ is a characteristic function or has mass $m$. Our strategy is to extend the functional $\Upsilon_{\mathrm{set}}$ as a functional $\Ups$ defined on $L^1(\R^d, [0,1])$. Applying the bathtub principle (see Proposition~\ref{prop_bathtub}) to a maximiser of the relaxed problem, we show that the supremum in~\eqref{maxE} is actually reached (see Corollary~\ref{coro_maxE}). This relaxation approach is not new: it was successfully applied to several variational problems in the last few years (see for instance~\cite{CiDeNoPo15, BonKnuRog, PegSanXia,BuChoTop}).
	
	\medskip 
	
	Given $f \in L^1(\R^d, \, [0,1])$, the set of admissible exterior transport plans is defined as
	\begin{equation*}
		\PI_f :=\lt\{\gamma\in\mathcal{M}_+(\R^d \times \R^d) : \gamma_x = f ,\,  \gamma_y \le 1-f\rt\}.
	\end{equation*}
	Here, the measures $f\,dx$ and $(1-f)\,dy$ are identified with their respective densities and $\gamma_x$ and $\gamma_y$ denote respectively the first and second marginals of $\gamma$. We then define the primal problem
	\[
		\Ups(f) :=\inf \lt\{\int c\,d\gamma : \gamma\in \PI_f\rt\}.
	\]
	We have $\Ups(\chi_E)=\Ups_{\mathrm{set}}(E)$ under mild assumptions on $c$ (see Theorem~\ref{thm_saturation}). Given $m>0$, our maximisation problem is now
	\begin{equation}\label{MP}
		\E(m):=\sup \lt\{ \Ups(f) : f \in L^1(\R^d, [0,1]), \,\int f\,dx=m\rt\}.
	\end{equation}
	By abuse of notation and when no confusion is possible, we refer to the variational problems by the values they attain (\textit{e.g.} we write $\Ups_{\mathrm{set}}(E)$ for~\eqref{Upsfunc}).
	
	\subsection{Main results}\hfill
	
	The first important result of this article is that maximisers of $\E(m)$ exist whenever $c$ is of the form $c(x,y) = k(y-x)$ for some $k: \R^d \to \R_+$ and satisfies
	
	\begin{enumerate}[(H1)]
		\item\label{cont}  $ k \in C(\R^d, \R_+)$, $k(0) = 0$ and $k(x) \to \infty$ as $|x| \to \infty$,
		\item\label{cone} $\forall x \not = 0$, 
		\[ \limsup_{r \to 0} \frac{1}{ r^d} \big|B_r(x) \cap \{y \in \R^d, \, k(y) < k(x)\} \big| > 0,
		\]
		\item\label{monot}  $\forall \, \sigma \in \mathbb{S}^{d-1}$, $r \mapsto k(r\sigma)$ is increasing on $\R_+$.
	\end{enumerate}
	Notice that $k$ is not assumed to be strictly convex, so that our results hold in cases where the existence of an optimal transport map is not guaranteed. Also observe that all the costs of the form $k(z) = |z|^p$ with $0 < p < \infty$ satisfy the above hypotheses. However, radial symmetry is not required and the costs $k(z)= |z|^p \,h(z/|z|)$ with $h$ positive and Lipschitz continuous on $\mathbb{S}^{d-1}$ are also admissible.
	\begin{theorem}\label{thm_MP_max}
		Assume that $c(x,y) = k(y-x)$ for $x,y\in \R^d$ with $k$ satisfying~(H\ref{cont}),(H\ref{cone})\&(H\ref{monot}). Then, for any $m>0$ the supremum in $\E(m)$ is attained. Moreover, there exists $R_* = R_*(m)$ such that (up to translation) any maximiser is supported in the ball $\overline B_{R_*}$.
	\end{theorem}
	
	Once the existence of maximisers for $\E(m)$ is established, the bathtub principle (see Proposition~\ref{prop_bathtub}) and a saturation result (see Theorem~\ref{thm_saturation}) imply that~\eqref{maxE} admits solutions.
	
	\begin{corollary}\label{coro_maxE}
		Assume that $c$ satisfies the hypotheses of Theorem~\ref{thm_MP_max}. Then, \eqref{maxE} admits a maximiser and $\E_{\mathrm{set}}(m) = \E(m)$ for any $m>0$.
	\end{corollary}
	
	As a  second main result, we establish that if $k$ is furthermore radially symmetric then $\E(m)$ and $\E_{\mathrm{set}}(m)$ are uniquely maximised by balls of volume $m$.
	\begin{theorem}\label{thm_ball_unique_max}
		Assume that $c(x,y)= k(|y-x|)$ for some $k \in C(\R_+, \R_+)$ increasing and such that $k(0) = 0$ and $k(x) \to \infty$ as $x \to \infty$. Then, for any $m>0$, the maximisers of $\E(m)$ (and consequently those of $\E_{\mathrm{set}}(m)$) are  the balls of volume $m$. Moreover the minimizer of $\Ups_{\mathrm{set}}(B_1)$ is the annulus $B_{2^{1/d}}\backslash B_1$.
	\end{theorem}
	We point out that cost functions satisfying the hypotheses of Theorem~\ref{thm_ball_unique_max} also satisfy hypotheses (H\ref{cont}),(H\ref{cone})\&(H\ref{monot}).
	Let us briefly sketch the proofs of these three results. They all strongly rely on the properties of the dual problem
	\begin{equation*}
		\Ups^*(f):=\sup\lt\{\int \lt(f \vhi + (1-f) \psi \rt) \, dx :  (\vhi,\psi) \in \Phi \rt\},
	\end{equation*}
	
	\noindent
	where 
	\[
	\Phi := \lt\{(\vhi, \psi) \in C_b(\R^d) \times C_b(\R^d), \, \psi \leq 0, \, \vhi(x)+\psi(y) \leq c(x,y) \ \forall \, (x,y) \in \R^d \times \R^d \rt\}.
	\] 
	
	\medskip
	
	We establish Theorem~\ref{thm_MP_max} using the direct method of Calculus of Variations. The main difficulty is to establish compactness of maximising sequences. If we refer to the concentration-compactness principle~\cite{PLLconccomp}, we have to prove that given a maximising sequence $f_n$, no mass escapes at infinity.  To do so we establish two crucial results. The first one is that $m \mapsto \E(m)/m$ is increasing (see Proposition~\ref{prop_mu_inc}). This implies that $m \mapsto \E(m)$ is strictly superadditive, \textit{i.e.} that for $m>m'>0$,
	\begin{equation}\label{stricsubadd}
		\E(m') + \E(m-m')< \E(m).
	\end{equation}
	Notice that this is the counterpart of the strict subbadditivity inequality (also called binding inequality) which is known to provide compactness in minimisation problems, see \textit{e.g.}~\cite{PLLconccomp,FrankLieb,frankNam}.
	Using the dual formulation $\Ups^*$ of $\Ups$, we obtain the second crucial result for Theorem~\ref{thm_ball_unique_max}: a monotonicity principle on the sum of marginals of minimisers $\gamma$ of $\Ups(f)$ (see Corollary~\ref{coro_monot_f}). This is the most delicate part of the proof. Combining this  and~\eqref{stricsubadd}, we prove that if $f$ is almost maximising then most of its mass must remain in a bounded region (see Proposition~\ref{prop_strong_tightness}). This  gives  tightness of maximising sequences for $\E(m)$.
	
	\medskip
	
	To prove Corollary~\ref{coro_maxE}, we consider a maximiser $f$ of $\E(m)$ provided by Theorem~\ref{thm_MP_max} and a pair of potentials $(\vhi, \psi)$ optimal for the dual problem $\Ups^*(f)$. Using the definition of $\Ups^*$ we see that $f$ is a maximiser of
	\[
	\sup\lt\{\int \widetilde{f} (\vhi - \psi)\, : 0 \leq \widetilde{f} \,\leq 1, \, \int \widetilde{f}\, = m\rt\}.
	\]
	By the bathtub principle, $f = \chi_{\{\vhi-\psi > t\}} + \theta$ for some $t \in \R$ and some $\theta \in L^1(\R^d, [0,1])$ supported in $\{\vhi-\psi = t \}$. Then for any measurable subset $G\subset  \{\vhi - \psi = t\}$ with $|G|=\int\theta\,$, the characteristic function of  $E:=\{\vhi-\psi > t\}\cup G$ is also a maximiser for $\E(m)$. By Theorem~\ref{thm_saturation} and Corollary~\ref{coro_unique_g} applied to $E$, there exists $F \subset \R^d$ such that any minimiser $\gamma$ of $\Ups(\chi_E)$ satisfies $\gamma_y = \chi_F$. This finally implies that $E$ maximises~\eqref{maxE}.
	
	\medskip	
	
	Regarding Theorem~\ref{thm_ball_unique_max}, as explained in Section~\ref{ballisunique}, we may assume without loss of generality that $m = \om_d$, the volume of the unit ball. Combining Theorem~\ref{thm_MP_max} and Lemma~\ref{lem_ctransfglobal} yields that
	\begin{equation}\label{doublesup}
		\sup_{\smallint f\, = \om_d} \lt\{ \sup_{(\psi^c, \psi) \in \Phi} \lt\{ \int f (\psi^c - \psi)\, + \int \psi\, \rt\}\rt\},
	\end{equation}
	coincides with $\E(m)$ and admits a solution $(f, \psi^c, \psi)$, where $\psi^c$ is the $c$-transform of $\psi$ (see Definition~\ref{def_ctransf}). To show that balls are maximisers of $\E(m)$, we establish that each term in~\eqref{doublesup} is improved by replacing $f$ by $\chi_{B_1}$ and $\psi$ by its symmetric increasing rearrangement $\psi_*$ (see Definition~\ref{def_sym_rear}). As $\smallint \psi\, = \smallint \psi_*\,$, the third term in~\eqref{doublesup} does not change under rearrangement. Regarding the second term, combining the Hardy-Littlewood inequality (see~\cite[Theorem 3.4]{LiebLoss}) and the bathtub principle yields (recall that $\psi \leq 0$)
	\[
	-\int f \psi\, \leq -\int f^* \psi_*\, \leq -\int \chi_{B_{1}} \psi_*\,,
	\] 
	where $f^*$ is the symmetric decreasing rearrangement of $f$ (see Definition~\ref{def_sym_rear}). The study of the first term $\smallint f \psi^c\,$ is more involved. Indeed it requires to understand the interactions between the operations of $c-$transform and symmetrization. To the best of our knowledge, this type of questions have not been addressed so far. Using the Brunn-Minkowski inequality, we obtain the following crucial comparison:
	\[
	(\psi^c)^* \leq (\psi_*)^c.
	\] 
	Combining this inequality with the Hardy-Littlewood inequality yields
	\begin{equation}\label{introHL}
		\int f \psi^c\, \leq \int f^* (\psi^c)^*\, \leq \int f^* (\psi_*)^c\,.
	\end{equation}
	Additionally, as $(\psi_*)^c$ is non-increasing, $\chi_{B_1}$ is a maximiser of 
	\begin{equation}\label{introbathtub}
		\sup \lt\{\int \widetilde{f} (\psi_*)^c\,: 0 \leq \widetilde f \leq 1, \, \int \widetilde{f}\, = \om_d \rt\},
	\end{equation}
	so that $\smallint f^* (\psi_*)^c\, \leq \smallint (\psi_*)^c \chi_{B_{1}}\,$. Lastly, by~\eqref{introHL}, $\smallint f \psi^c\,\leq \smallint \chi_{B_{1}} (\psi_*)^c\,$. This eventually proves that unit balls maximise $\E(\om_d)$.
	
	As for uniqueness, the key property to establish is that $(\psi_*)^c$ is decreasing on $B_{1}$ (see Lemma~\ref{lem_monotpotball}). Indeed, by~\cite[Theorem 3.4]{LiebLoss}, this implies that $\chi_{B_{1}}$ is the unique maximiser of~\eqref{introbathtub}. Combining this with the fact that the inequalities in~\eqref{introHL} are now equalities, we obtain that $f^* = \chi_{B_{1}}$, so that $f = \chi_{E}$ for some $E \subset \R^d$. Using the equality case of the Brunn-Minkowski inequality, we then show that (up to a translation) $f = \chi_{B_{1}}$, concluding the proof.
	
	\subsection{Motivation}\hfill
	
	In~\cite{BCL2020}, the following variational problem was introduced: 
	\begin{equation}\label{singleOptim}
		\inf_{|E| = \om_d} \lt\{ P(E) + \alpha \Ups_p(E) \rt\},
	\end{equation}
	where $\alpha >0$ and where $\Ups_p$ is the functional $\Ups$ defined in~\eqref{Upsfunc} with the cost $c(x,y) = |x-y|^p$. Such a variational problem may be used to model the formation of bi-layer biological membranes (see~\cite{PeRo09,LuPeRo14}).  Existence of minimisers were obtained in the series of work~\cite{BCL2020,XiaZhou,NoToVenk,CanGol}.
	
	\medskip
	
	Notice that~\eqref{singleOptim} is an isoperimetric problem with a non-local term $\Ups_p$ where the perimeter term favors aggregation and the  non-local term is of repulsive nature. One of the best-known examples of this type of problem is Gamow's liquid drop model for the atomic nucleus. Since the beginning of the 2010s (see~\cite{gamowhist} for an historical perspective), this model has received a lot of attention from the mathematical community, and several versions of it have been studied, see for instance~\cite{KnMu, golnovruf, KnMuNov,GolMerPe}. In this framework, the functional to be optimized is  
	\[
	 P(E) + \alpha V_\beta(E),
	\]
	where the repulsive non-local term $V_\beta(E)$ is given by the Riesz potential
	\begin{equation}\label{Rieszpot}
	V_\beta(E) := \int_E\int_E \frac{dxdy}{|x-y|^{d-\beta}},\qquad \text{with }\beta \in (0,d).
	\end{equation}
	It is well known that the \emph{minimizers} of the perimeter under volume constraints are the balls of the given volume: this reflects the aggregative nature of the perimeter. A natural way to illustrate the competition  between the perimeter and the non-local term is then to establish that, on the contrary, the \emph{maximizers} of the non-local term under volume constraints are balls. 
	\begin{enumerate}[--]
	\item In the case of~\eqref{Rieszpot}   this is a consequence of Riesz's rearrangement inequality.
	\item For the non-local functional $\Ups_p$ and more generally for the functionals $\Ups$ (and $\Ups_{\mathrm{set}}$), this corresponds to Theorem~\ref{thm_ball_unique_max}. Assuming some natural hypotheses on the cost $c$, the \emph{maximizers} of $\Ups$ under volume constraint  are the characteristic functions of balls. The theorem applies in particular to $\Ups_p$. 
	\end{enumerate}
	 In the latter case, the proof is much more involved since the rearrangement argument does not seem to work well for the primal problem. We consider instead the dual problem  $\Ups^*$ and study  the subtle and fortunately favourable interplay between rearrangements and $c-$transforms.\label{Ups(f)}
	\medskip
	
	As a closing remark, we point out that the functional $\Ups$ is a particular case of  the optimal partial transport problem studied in~\cite{FigPartial, DePMSV}.
	\subsection*{Added after submission:} A few days after this paper has been submitted, Burchard, Carazzato and Topaloglu posted on the Arxiv a paper proving very similar results to ours but with totally different methods, see \cite{BuCaTop}.

	\subsection{Organization of the article}\hfill
	
	The paper is structured as follows. In Section~\ref{StdOT}, we introduce the notation and review standard facts related to optimal transport in complete separable metric spaces. In Section~\ref{Xcompact}, we obtain preliminary results on the functional $\Ups$ defined  in compact spaces. In Section~\ref{trslinv}, we establish Theorem~\ref{thm_MP_max}. Eventually, in Section~\ref{ballisunique}, we prove Theorem~\ref{thm_ball_unique_max}.
	
	\section{Notation and preliminary results}\label{StdOT}
	
	\subsection{Notation}\hfill
	
	Let $(X, d_X)$ be a Polish space endowed with a positive Radon measure $\lambda$.
	
	\medskip
	
	Given a function $f : X \to \R$, we decompose it as: 
	\[
	f = f_+ + f_- \quad \text{ with } \quad f_+ := \max(0, f) := 0 \vee f \quad\ \text{and}\ \quad f_- := \min(0, f) := 0 \wedge f.
	\]
	Let us stress that $f_-$ is non-positive, contrary to the classical decomposition of a function into its positive and negative parts.
	
	\medskip
	
	We endow $\mathcal{M}_+(X)$ with the topology  induced by duality with $C_b(X)$ (often called narrow convergence). The convergence of a sequence $\mu_n \in \mathcal{M}_+(X)$ to $\mu \in \mathcal{M}_+$  is written: $\mu_n \st*\rightharpoonup \mu$ as $n \to \infty$.
	
	\medskip
	
	Given a measure $\mu \in \mathcal{M}_+(X)$ and a set $A \subset X$, the restriction of $\mu$ to $A$ is the measure $\mu \restr A$ defined as $\mu \restr A (B) := \mu (B \cap A)$ for every Borel set $B$ of $X$. The support of $\mu$, denoted by $\supp \mu$, is the closed set defined by
	\[
	\supp \mu := \lt\{x \in X : \mu(A) > 0 \text{ for all open set A containing } x\rt\}.
	\]
	Given $f \in L^1(X, \lambda)$ the support of $f$ is defined as the support of the measure $f d\lambda$ and denoted by $\supp f$. We identify the measure $f d\lambda$ with its density $f$ and write $f_n\st*\rightharpoonup f$ as $n \to \infty$ to signify that $\int f_n\xi\,$ converges to $\int f\xi\,$ for every $\xi \in C_b(X)$.
	
	\medskip
	
	Given a function $f \in L^1_{\mathrm{loc}}(\R^d, \R)$, we denote by $\mathrm{Leb}(f)$ the set of its Lebesgue points.
	
	\medskip
	
	Given $x \in \R^d$ and $r >0$, $B_r(x)$ denotes the open ball of radius $r$ centred at $x$, and $B_r$ denotes the open ball of radius $r$ centred at $0$. The closed ball of radius $r$ centred at $x$ is denoted by $\overline{B}_r(x)$. The volume of the unit ball in $\R^d$ is denoted by $\om_d$.
	
	\medskip
	
	Given two sets $A, B$ of $\R^d$, we define their sum $A+B:= \{a +b, \, a \in A, \, b \in B\}$. The gap between $A$ and $B$ is $d(A,B) := \inf \{|a-b|, \, a \in A, \, b \in B\}$.
	
	\subsection{Optimal transport theory}\hfill
	
	In this subsection, we recall some results regarding standard optimal transport theory. Most of the material presented here comes from~\cite[Chapter 1]{OTAM}.
	\medskip
	
	Let $(X, d_X)$ be a complete separable metric space (\ie a Polish space) and let $c : X \times X \to \R$ be measurable. Given $\mu, \nu \in \mathcal{M}_+(X)$ such that $\mu(X) = \nu(X)$, the Kantorovitch problem with marginals $\mu$ and $\nu$ and cost $c$ is
	\begin{equation}\label{KP}
		\mathcal{T}_c(\mu, \nu) := \inf \lt\{\int c \, d\gamma : \, \gamma \in \PI(\mu, \nu)\rt\},
	\end{equation}
	where $\PI(\mu, \nu)$ is the set of transport plans between $\mu$ and $\nu$, \ie
	\begin{equation*}
		\PI(\mu, \nu) := \lt\{ \gamma \in \mathcal{M}_+(X \times X) : \, \gamma_x = \mu, \,  \gamma_y = \nu \rt\}.
	\end{equation*}
	Problem~\eqref{KP} admits a dual formulation given by
	\begin{equation}\label{DKP}
		\mathcal{T}_c^*(\mu, \nu) := \sup \lt\{\int \vhi \, d\mu +\int \psi \, d\nu : \, \vhi, \psi \in C_b(X), \, \vhi \oplus \psi \leq c\rt\},
	\end{equation}
	where the function $\vhi \oplus \psi$ is defined on $X \times X$ by $(\vhi \oplus \psi)(x,y):= \vhi(x)+\psi(y)$.
	
	\medskip

	\begin{theorem}[Theorem 1.7 of~\cite{OTAM}]
		Let $c : X \times X \to \R$ be lower semi-continuous and bounded from below and let $\mu, \nu \in \mathcal{M}_+(X)$ with $\mu(X) = \nu (X)$. Then~\eqref{KP} admits a solution and
		\begin{equation*}
			\mathcal{T}_c(\mu, \nu) = \mathcal{T}_c^*(\mu, \nu).
		\end{equation*}
	\end{theorem}
	Using the notion of $c$-transform of a function, the maxima of~\eqref{DKP} can be further characterised.
	
	\begin{definition}\label{def_ctransf}
		Given a function $\xi : X \to \R \cup \{+\infty\}$, we define its $c$-transform (or $c$-conjugate) $\xi^c : X \to \R \cup \{-\infty\}$ by
		\begin{equation*}
			\xi^c(y) := \inf_{x \in X} \lt\{c(x,y) - \xi(x) \rt\}.
		\end{equation*}
		
		\noindent
		Denoting $\bar c(y,x) := c(x,y)$, the $\bar c$-transform of $\zeta : X \to \R \cup \{+\infty\}$ is given by
		
		\begin{equation*}
			\zeta^{\bar c}(x) := \inf_{y \in X} \lt\{\bar c(y,x) - \zeta(y) \rt\}.
		\end{equation*}
		
		\noindent
		A function $\psi : X \to \R \cup \{-\infty\}$ is said to be $\bar c$-concave if there exists $\xi :X \to \R \cup \{+\infty\}$ such that $\psi = \xi^c$ (the definition of $c$-concavity is analogous).
	\end{definition}
	
	\begin{definition}
		Let $(X, d_X)$ be a metric space and $\om \in C(\R_+, \R_+)$ be increasing and such that $\om(0) = 0$. A function $\vhi : X \to \R$ is $\om$-continuous if for all $x, x' \in X$,
		\[
		|\vhi(x) - \vhi(x')| \leq \om (d_X(x, x')).
		\]
		
		\noindent
		Similarly, we say that $c : X \times X \to \R$ is $\om$-continuous if for all $x, x', y, y' \in X$,
		\[
		|c(x,y)-c(x', y')| \leq \om(d_X(x,x')+d_X(y,y')).
		\]		
	\end{definition}
	
	\begin{proposition}\label{prop_ctransf}
		Let $\vhi, \psi : X \to \R$ be fixed and assume that $\vhi^c$ and $\psi^{\bar c}$ take real values. The following statements hold:
		\begin{enumerate}[(i)]
			\item If $c$ is $\om$-continuous, then $\vhi^c$ is also $\om$-continuous,
			\item $\vhi^{c \bar c} \geq \vhi$, and $\vhi^{c \bar c} = \vhi$ if and only if $\vhi$ is $c$-concave,
			\item $\vhi^c$ is the largest function $\psi$ compatible with the constraint $\vhi \oplus \psi \leq c$ and $\psi^{\bar c}$ is the largest function $\vhi$ compatible with the constraint $\vhi \oplus \psi \leq c$.
		\end{enumerate}
		
	\end{proposition}
	Remark that if $X$ is compact and $\vhi$, $\psi$ and $c$ are bounded then $\vhi^c$ and $\psi^{\bar c}$ take real values. Moreover, if $c$ is continuous, say $\om$-continuous, the proposition states that $\vhi^c$ and $\psi^{\bar c}$ are $\om$-continuous. This yields the following existence result for~\eqref{DKP}.
	
	\begin{theorem}[Proposition 1.11 of~\cite{OTAM}]
		Let $X$ be a compact metric space and $c : X \times X \to \R$ be continuous. Then there exists a solution $(\vhi, \psi)$ to~\eqref{DKP}, where $\vhi$ is $c$-concave and $\psi = \vhi^c$. In particular,
		\begin{equation*}
			\mathcal{T}_c^*(\mu, \nu) = \max \lt\{\int \vhi \, d\mu + \int \vhi^c \, d\nu :\, \vhi \, \text{ c-concave} \rt\}.
		\end{equation*}
	\end{theorem}
	A pair of functions maximising~\eqref{DKP} is called a pair of Kantorovitch potentials.
	
	\section{Study of $\Ups$ in compact metric spaces}\label{Xcompact}
	
	Let $(X, d_X)$ be a compact metric space and let $c : X \times X \to \R$ be a continuous cost function. We endow $(X, d_X)$ with a measure $\lambda \in \mathcal{M}_+(X)$ such that $\lambda(X)>0$ and denote by $L^1(X)$ the set of $\R$-valued functions integrable with respect to $\lambda$. Given $f \in L^1(X)$, we define the set of admissible transport plans
	\[
	\PI_f :=\lt\{\gamma\in\mathcal{M}_+(X\times X) : \gamma_x = f,\, \gamma_y \le 1-f \rt\}
	\]
	and the primal problem
	\begin{equation}\label{PPXcomp}
		\Ups(f) :=\inf \lt\{\int c\,d\gamma : \gamma\in \PI_f\rt\}.
	\end{equation}
	Notice that $\PI_f$ is empty whenever $f$ does not satisfy $0 \leq f \leq 1$  or when $\smallint f\,d\lambda > \lambda(X)/2$.  In the other cases, there exists $g \in L^1(X)$ such that $g\ge0$, $f+g\le1$ and $\smallint g \, d\lambda=\smallint f \,d\lambda$. Thus,
	\[
	\gamma:= \frac{1}{\smallint f \, d\lambda}(f \, d\lambda)\otimes (g \,d\lambda)\ \in\PI_f,
	\]
	and $\PI_f$ is not empty. We now fix $0 < m \leq \lambda(X)/2$ and define
	\[
	L^1_m := \lt\{f \in L^1(X, [0,1]) : \int f\, \leq m \rt\}.
	\]
	Given $f \in L^1_m$ and $\vhi, \psi \in C(X)$, we set
	\begin{equation}\label{Kf}
		K_f(\vhi, \psi) := \int \lt(f \vhi + (1-f) \psi \rt) \, d\lambda
	\end{equation}
	and define the dual problem
	\begin{equation}\label{DPXcomp}
		\Ups^*(f):=\sup\lt\{K_f(\vhi, \psi) :  (\vhi, \psi) \in \Phi \rt\},
	\end{equation}
	where
	\[
	\Phi := \lt\{(\vhi, \psi) \in C(X) \times C(X), \, \psi \leq 0, \, \vhi \oplus \psi \leq c \rt\}.
	\]
	
	\medskip
	For the remainder of the section we fix $f \in L^1_m$. As in the classical theory of optimal transport, a simple  application of the direct method of Calculus of Variations shows that~\eqref{PPXcomp} admits a minimiser.
	
	\begin{proposition}\label{prop_exis_gamma_opt}
		Assume that $X$ is a compact metric space and that $c \in C(X \times X, \R)$. Then, the infimum in~\eqref{PPXcomp} is a minimum.
	\end{proposition}
	
	\begin{remark}\label{rem_KP_Ups}
		If we let $f \in L^1_m$, by Proposition~\ref{prop_exis_gamma_opt}, there exists $\gamma \in \mathcal{M}_+(X \times X)$ optimal for $\Ups(f)$. Notice that $\gamma$ solves the classical optimal transport problem from $f$ towards $g := \gamma_y$ defined by~\eqref{KP}. Moreover, we have the identity $\Ups(f) = \mathcal{T}_c(f,g)$.
	\end{remark}
	
	Let us now show that $\Ups^*(f)=\Ups(f)$ and that~\eqref{DPXcomp} admits a maximising pair $(\vhi,\psi)$. We first establish that we can reduce the set of competitors for~\eqref{DPXcomp}. To simplify the notation we denote by $\vhi^c_{\,-}$ the function $(\vhi^c)_- := \vhi^c \wedge 0$.
	
	\begin{lemma}\label{lem_Phi}
		Assume that $X$ is a compact metric space and that $c \in C(X \times X, \R)$. Then, there holds
		\begin{equation}\label{DPp}
			\Ups^*(f)=\sup\,\{K_f(\psi^{\bar c}, \psi) : \psi = \vhi^c_{\,-} \text{ for some } \vhi \in \Phi' \},
		\end{equation}
		
		\noindent
		where
		\begin{equation}\label{defPhiprime}
			\Phi' := \lt\{\vhi \in C(X), \, \vhi = (\vhi^c_{\,-})^{\bar c}, \; \max \vhi^c \geq 0 \rt\}.
		\end{equation}
	\end{lemma}
	
	\medskip
	
	\begin{proof}~\\
		\textit{Step 1. We can replace $\psi$ by $\vhi^c_{\,-}$ and assume that $\max \vhi^c \geq 0$.}
		
		Let $(\vhi,\psi)\in \Phi$. By Proposition~\ref{prop_ctransf} $(iii)$, $\psi\le\vhi^c$, so that $\psi\le\vhi^c\wedge0=\vhi^c_{\,-}$. As $1-f\ge0$, $K_f(\vhi,\vhi^c_{\,-})\ge K_f(\vhi,\psi)$. Therefore, we can restrict the maximisation to the pairs $(\vhi,\vhi^c_{\,-})$ in the supremum~\eqref{DPXcomp}. Now, if $\max \vhi^c=-t<0$ we set $\widetilde \vhi:=\vhi-t$ so that ${\widetilde\vhi}^c=\vhi^c+t$. Consequently, $\max\widetilde\vhi^c=0$ and in particular, ${\widetilde\vhi}^c_{\,-}={\widetilde\vhi}^c$ so that $({\widetilde\vhi},{\widetilde\vhi}^c)\in\Phi$. We then compute
		\begin{align*}
			K_f(\widetilde\vhi,\widetilde\vhi^c_{\-})
			&=\int f(\widetilde\vhi -\widetilde\vhi^c)\,d\lambda+\int \widetilde\vhi^c\,d\lambda\\
			&\ge\int f(\vhi -\vhi^c)\,d\lambda+\int \vhi^c\,d\lambda + t(\lambda(X) -2m)\\
			&=K_f(\vhi,\vhi^c_{\,-})+ t(\lambda(X)-2m).
		\end{align*}
		
		\noindent
		As $2m \leq \lambda(X)$ we obtain $K_f(\widetilde\vhi,\widetilde\vhi^c_{\,-}) \geq K_f(\vhi,\vhi^c_{\,-})$. Hence
		\begin{equation*}
			\Ups^*(f)=\sup \lt\{K_f(\vhi,\vhi^c_{\,-}) : \vhi\in C(X), \, \max \vhi^c \geq 0 \rt\}.
			\smallskip
		\end{equation*}
		
		\noindent
		\textit{Step 2. There holds $\vhi = (\vhi^c_{\,-})^{\bar c}$.} 
		
		Let us introduce the mapping $P:C(X)\to C(X)$ defined by $P(\vhi):=(\vhi^c_{\,-})^{\bar c}$. For $\vhi\in C(X)$, $\vhi^c\ge\vhi^c_{\,-}$, so that $P(\vhi)=(\vhi^c_{\,-})^{\bar c}\ge \vhi^{c \bar c}$. By Proposition~\ref{prop_ctransf} $(ii)$, $\vhi^{c \bar c} \geq \vhi$, hence
		\begin{equation}\label{proof_lem_Phi_1}
			P(\vhi)\ge \vhi.
		\end{equation}
		
		\noindent
		By Proposition~\ref{prop_ctransf} $(ii)$ again (but applied to $\bar c c$ instead of $c \bar c$), there holds $P(\vhi)^c=(\vhi^c_{\,-})^{\bar cc} \ge \vhi^c_{\,-}$. Taking the negative part yields
		\begin{equation}\label{proof_lem_Phi_2}
			P(\vhi)^c_{\,-} \geq  \vhi^c_{\,-}.
		\end{equation}
		
		\noindent
		We deduce from~\eqref{proof_lem_Phi_1} and~\eqref{proof_lem_Phi_2} that 
		\[
		K_f(P(\vhi),P(\vhi)^c_{\,-}) \ge K_f(\vhi,\vhi^c_{\,-}).
		\]
		
		\noindent
		Now, we observe that if $\max \vhi^c \geq 0$ we also have $\max \vhi^c_{\,-} = 0$ and, by~\eqref{proof_lem_Phi_2}, $\max P(\vhi)^c_{\,-} = 0$ which implies that $\max {P(\vhi)^c} \geq 0$. Hence,
		\begin{equation}\label{proof_lem_Phi_25}
			\Ups^*(f)=\sup \lt\{K_f(\widetilde \vhi, \widetilde \vhi^c_{\,-}) : \widetilde\vhi\in C(X), \, \max \widetilde \vhi^c \geq 0,\, \widetilde \vhi=P(\vhi)\text{ for some } \vhi\in C(X) \rt\}.
		\end{equation}
		
		\noindent
		To conclude, we show that $P(P(\vhi)) = P(\vhi)$ for any $\vhi \in C(X)$. By~\eqref{proof_lem_Phi_1}, $P(P(\vhi)) \geq P(\vhi)$. Taking the $\bar c$-transform in~\eqref{proof_lem_Phi_2} yields $P(P(\vhi))\leq P(\vhi)$ and we have indeed $P(P(\vhi))=P(\vhi)$. Hence we have $\tilde \vhi\in\Phi'$ in~\eqref{proof_lem_Phi_25} and we get 
		\begin{equation}\label{proof_lem_Phi_3}
			\Ups^*(f) = \sup\lt\{ K_f (\tilde\vhi, \tilde\vhi^c_{\,-}) : \tilde\vhi \in \Phi' \rt\}.
		\end{equation}
		
		\noindent
		Finally, by definition $\tilde\vhi = (\tilde\vhi^c_{\,-})^{\bar c}$ for $\tilde\vhi \in \Phi'$ and~\eqref{DPp} follows from~\eqref{proof_lem_Phi_3} by letting $\psi := \tilde\vhi^c_{\,-}$.
	\end{proof}
	
	We can now establish that the supremum in~\eqref{DPp} is reached.
	
	\begin{proposition}\label{prop_ex_DPp}
		Assume that $X$ is a compact metric space and that $c \in C(X \times X, \R)$. Then, the set $\Phi'$ is compact in $(C(X), \|\cdot\|_\infty)$ and the suprema in~\eqref{DPp} and~\eqref{DPXcomp} are attained.
	\end{proposition}
	\smallskip
	
	\begin{proof}~\\		
		Let us show that $\Phi'$ is compact. Let $\vhi_n$ be a sequence in $\Phi'$. The function $c$ is $\om$-continuous for some  modulus of continuity $\om\in C(\R_+,\R_+)$, so that by Proposition~\ref{prop_ctransf} $(i)$ for every $n \geq 0$,  $\vhi^c_n:=(\vhi_n)^c$ and ${\vhi_n^c}_-:=((\vhi_n)^c)_-$ are $\om$-continuous. By definition of $\Phi'$, $\vhi_n = ({\vhi_n^c}_-)^{\bar c}$, so that $\vhi_n$ is also $\om$-continuous for every $n \geq 0$. Let us show that the sequences $\vhi_n$ and ${\vhi_n^c}_-$ are uniformly bounded in $(C(X), \|\cdot\|_\infty)$. We observe that for every $n \geq 0$, $\max \vhi_n^c \geq 0$. In particular this implies $\max {\vhi_n^c}_-=0$. Denoting by $x_n$ a point of $X$ such that ${\vhi^c_n}_- (x_n) = 0$, by $\om$-continuity we have for $x \in X$ and $n \geq 0$,
		\[
		-\om(\diam(X))\leq -\om(d_X(x,x_n)) \leq  {\vhi_n^c}_-(x)-{\vhi_n^c}_-(x_n) = {\vhi_n^c}_-(x) \le 0.
		\]
		\noindent
		Thus the sequence ${\vhi_n^c}_-$ is uniformly bounded in $(C(X), \|\cdot\|_\infty)$. By definition of the $c$-transform
		\[
		\min_{X\times X} c -\max_X {\vhi_n^c}_- \le ({\vhi_n^c}_-)^{\bar c}  \le \max_{X\times X} c-\min_X {\vhi_n^c}_-.
		\] 
		
		\noindent
		Hence the sequence $\vhi_n$ is also uniformly bounded. By Arzel\'a-Ascoli's theorem, there exists a pair $(\vhi, \psi) \in C(X) \times C(X)$ such that, up to extraction of a subsequence, $(\vhi_n, {\vhi_n^c}_-)$ converges uniformly to $(\vhi, \psi)$.
		
		Let us show that $\vhi \in \Phi'$. By Proposition~\ref{prop_ctransf} $(iii)$ and by uniform convergence $\vhi_n^c \to \vhi^c$ as $n \to \infty$ so that
		\begin{equation}\label{ctransfunifconv}
			{\vhi_n^c}_- \to \vhi^c_{\,-} \quad \text{uniformly as } n \to \infty,
		\end{equation}
		
		\noindent
		which yields $\psi = \vhi^c_{\,-}$. From~\eqref{ctransfunifconv} and the uniform continuity of $c$, we deduce that 
		\[
		({\vhi_n^c}_-)^{\bar c} =\vhi_n \to (\vhi^c_{\,-})^{\bar c} \quad \text{ uniformly as } n \to \infty.
		\]
		
		\noindent
		Since $\vhi_n \to \vhi$ as $n \to \infty$, we obtain $\vhi = (\vhi^c_{\,-})^{\bar c}$. Lastly, by uniform convergence, the fact that $\max \vhi_n^c \geq 0$ for all $n \geq 0$ implies that $\max \vhi^c \geq 0$, so that $\vhi \in \Phi'$. This shows that $\Phi'$ is a compact subset of $(C(X), \|\cdot\|_\infty)$.
		
		\medskip
		
		Let now $\psi_n$ be a maximising sequence for~\eqref{DPp}. For all $n \geq 0$, there exists $\vhi_n \in \Phi'$ such that $\psi_n = {\vhi_n^c}_-$. By compactness of $\Phi'$, $\vhi_n \to \vhi$ as $n \to \infty$ for some $\vhi \in \Phi'$. Setting $\psi = \vhi^c_{\,-}$, we have $\psi_n \to \psi$ and $\psi_n^{\bar c} \to \psi^{\bar c}$ as $n \to \infty$. The functional $K_f$ being continuous with respect to uniform convergence, we obtain
		\[
		K_f(\psi^{\bar c}, \psi )=\lim K_f(\psi^{\bar c}_n, \psi_n) = \Ups^*(f).
		\]
		
		\noindent
		This proves that $\psi$ is a maximiser for~\eqref{DPp} and by Lemma~\ref{lem_Phi}, $\psi$ also maximises~\eqref{DPXcomp}.
	\end{proof}
	
	We are now ready to prove that there is no duality gap between~\eqref{PPXcomp} and~\eqref{DPXcomp}. The proof is an adaptation of~\cite[Section 1.6.3]{OTAM}.
	
	\begin{proposition}\label{prop_DP=OP}
		Assume that $X$ is a compact metric space and that $c \in C(X \times X, \R)$. Then,
		\begin{equation*}
			\Ups^*(f)= \Ups(f).
		\end{equation*}
	\end{proposition}
	
	\begin{proof}~\\
		\textit{Step 1. Definition of $H$ and first properties.} 
		
		For $p\in C(X\times X)$, we define
		\[
		H(p):=-\sup\lt\{\int \lt(f\vhi+(1-f)\psi \rt) \,d\lambda : (\vhi,\psi) \in \Phi_p \rt\}
		\]
		
		\noindent
		where
		\[
		\Phi_p := \lt\{(\vhi,\psi) \in C(X) \times C(X),\ \psi\le0,\ \vhi\oplus\psi\le c-p \rt\}.
		\]
		
		\noindent
		We first observe that $c-p$ is continuous and bounded from below. Thus, by applying Proposition~\ref{prop_ex_DPp} with $c-p$ in place of $c$, we see that the above supremum is a maximum.
		
		Let us now show that $H$ is convex. Let $p_0,p_1\in C(X \times X)$ and $\theta\in[0,1]$ and let us set $p:=(1-\theta)p_0+\theta p_1$. We denote by $(\vhi_0,\psi_0)$ and $(\vhi_1,\psi_1)$ two maximising pairs associated with $p_0$ and $p_1$ and set $\vhi:=(1-\theta)\vhi_0+\theta \vhi_1$, $\psi:=(1-\theta)\psi_0+\theta \psi_1$. We see that $(\vhi,\psi)$ is an admissible pair ($\psi \leq 0$ and $\vhi \oplus \psi \leq c-p$), so that
		\[
		H(p)\le -\int \lt(f\vhi+(1-f)\psi\rt)\,d\lambda  = (1-\theta)H(p_0) +\theta H(p_1).
		\]
		
		\noindent
		This proves that $H$ is convex.

		Next, we establish that $H$ is lower semi-continuous in $(C(X\times X),\|\cdot\|_\infty)$. Let $p_n$ and $p$ be elements of $C(X\times X)$ such that $p_n\to p$ uniformly as $n \to \infty$. The sequence $c-p_n$ is uniformly equi-continuous. Therefore, proceeding as in the proof of Proposition~\ref{prop_ex_DPp}, there exists a sequence of uniformly bounded and equi-continuous admissible pairs $(\vhi_n, \psi_n)$ such that 
		\[
		H(p_n)=-\int \lt(f\vhi_n+(1-f)\psi_n\rt)\,d\lambda\qquad\text{for every }n \geq 0.
		\]
		
		\noindent
		We first extract a subsequence $p_{n'}$ such that $\lim_{n'} H(p_{n'})=\liminf_n H(p_n)$.
		By Arzelà-Ascoli's theorem, there exists $(\vhi, \psi) \in C(X) \times C(X)$ such that $\vhi_{n'}\to\vhi$ and $\psi_{n'}\to\psi$ uniformly as $n'\to \infty$. By pointwise convergence, $\psi\le0$ and $\vhi\oplus\psi\le c-p$. Passing to the limit yields
		\[
		H(p)\le -\int \lt(f\vhi+(1-f)\psi\rt)\,d\lambda  = -\lim_{n'} \int \lt(f\vhi_{n'}+(1-f)\psi_{n'}\rt)\,d\lambda = \liminf_n H(p_n).
		\]
		
		\noindent
		Thus $H$ is lower semi-continuous.
		
		\medskip
		
		\noindent
		\textit{Step 2. Absence of duality gap.}
		
		Since $H$ is convex and lower semi-continuous on the Banach space $(C(X\times X),\|\cdot\|_\infty)$, we have  $H(0)=H^{**}(0)$. Here, for a Banach space $\mathcal{X}$ and a function $F:\mathcal{X}\to\R\cup\{+\infty\}$, $F^*$ denotes the Legendre transform of $F$ defined on the topological dual $\mathcal{X}^*$ of $\mathcal{X}$ by
		\[
		F^*(x^*):=\sup \lt\{ x^*(x)-F(x) : x\in \mathcal{X}\rt\}.
		\]
		
		\noindent
		In particular,
		\begin{equation}\label{proof_prop_DP=OP_1}
			\Ups^*(f)=-H(0)=-H^{**}(0) = \inf\{ H^*(\gamma) : \gamma\in \mathcal{M}(X\times X)\}.
		\end{equation}
		
		\noindent
		We now compute $H^*$. Let $\gamma\in \mathcal{M}(X\times X)$. By definition,
		\begin{equation*}
			H^*(\gamma)=\sup_{p \in C(X\times X) } \lt\{\int p\,d\gamma + \sup_{(\vhi, \psi) 
				\in \Phi_p}\lt\{ \int \lt(f\vhi+(1-f)\psi\rt)\,d\lambda\rt\}\rt\}.
		\end{equation*}
		
		\noindent
		Let us first assume that there exists $q\in C(X\times X,\R_+)$ such that $t:=-\int q\,d\gamma >0$. We set $\vhi=\min c$, $\psi=0$ and $p_n:=-nq$ for $n\ge 1$. We obtain
		\[
		H^*(\gamma)\ge n t - m |\min c|\ \to \infty \qquad \text{ as } n \to \infty.
		\]
		
		\noindent
		Thus, when computing $H^*(\gamma)$, we may assume that $\gamma \geq 0$. We rewrite $H^*(\gamma)$ as 
		\begin{equation}\label{Hstar}
			\begin{aligned}
				H^*(\gamma) = \int c \, d\gamma + \sup_{p \in C(X \times X)} \sup_{(\vhi, \psi) \in \Phi_p} &\bigg\{\int (p-c + \vhi \oplus \psi) \, d\gamma \\ &+\int \vhi \, d (f \lambda - \gamma_x) + \int \psi \, d((1-f) \lambda - \gamma_y))\bigg\}.
			\end{aligned}
		\end{equation}
		
		\noindent
		Let us set
		\begin{equation*}
			G(\gamma) := \sup\lt\{\int \vhi \, d(f \lambda - \gamma_x) + \int \psi \, d((1-f)\lambda - \gamma_y) :
			(\vhi,\psi) \in C(X) \times C(X),\ \psi\le0 \rt\}.
		\end{equation*}
		
		\noindent
		On the one hand, given $(\vhi, \psi) \in \Phi_p$ and $\gamma \ge 0$, 
		\[
		\int (p - c + \vhi \oplus \psi)\,d\gamma \leq 0.
		\]
		
		\noindent
		Therefore, $H^*(\gamma) \leq \smallint c \, d \gamma + G(\gamma)$. On the other hand, given $(\vhi, \psi)$ admissible for $G(\gamma)$, setting $p = c - \vhi \oplus \psi$ yields the converse inequality thanks to~\eqref{Hstar}. Hence
		\begin{equation}\label{Hstar2}
			H^*(\gamma)=\int c\, d\gamma + G(\gamma).
		\end{equation}
		
		\noindent
		Given $\gamma \in \mathcal{M}_+(X \times X)$, we have $G(\gamma) = 0$ if $\gamma \in \PI_f$ and $G(\gamma) = +\infty$ otherwise. Combining this with~\eqref{Hstar2}, we obtain that for $\gamma \in \mathcal{M}(X \times X)$,
		\begin{equation*}
			H^*(\gamma)= \begin{cases} \ds\int c\, d\gamma &\text{if }\gamma\in \PI_f,\\
				+\infty&\text{in the other cases.}\end{cases}
		\end{equation*}
		
		\noindent
		Taking the infimum with respect to $\gamma\in \mathcal{M}(X\times X)$ and recalling~\eqref{proof_prop_DP=OP_1}, we get
		\[
		\Ups^*(f)= \inf \lt\{H^*(\gamma):\gamma\in \mathcal{M}(X\times X)\rt\} = \inf \lt\{\int c\,d\gamma :\gamma\in\PI_f\rt\}=\Ups(f),
		\]
		which concludes the proof.
	\end{proof}	
	
	\begin{remark}
		There is still no duality gap between~\eqref{PPXcomp} and~\eqref{DPXcomp} if we only assume $c$ to be lower semi-continuous. This result can be obtained by approximating $c$ pointwise from below by a non-decreasing sequence of continuous functions.
	\end{remark}
	
	In the remainder of the section, we focus on the properties of the potentials $(\psi^{\bar c}, \psi)$ maximising~\eqref{DPp}. We first show that the sign of $\psi^{\bar cc}$ enforces constraints on the local values of the marginals of any plan $\gamma$ optimal for $\Ups(f)$.
	
	\begin{proposition}\label{prop_Ups_OT}
		Assume that $X$ is a compact metric space and that $c \in C(X \times X, \R)$. Let  $(\psi^{\bar c}, \psi)$ be a maximiser of~\eqref{DPp} and let $\gamma$ be a minimiser of~\eqref{PPXcomp}. We set $g := \gamma_y$. Then, $\gamma$ is a minimiser of~\eqref{KP} with $(\mu, \nu) = (f,g)$ and $(\psi^{ \bar c}, \psi^{\bar c c})$ is a pair of Kantorovitch potentials realising the maximum in~\eqref{DKP}. Moreover, up to $\lambda$-negligible sets,
		\begin{equation}\label{psi_KP}
			f+g\equiv1\quad\text{on }\{\psi^{\bar cc} < 0 \}\qquad\text{and}\qquad g\equiv0\quad\text{on }\{\psi^{\bar cc}>0\}.
		\end{equation}
	\end{proposition}
	\smallskip
	
	\begin{proof}~\\
		By Remark~\ref{rem_KP_Ups}, $\gamma$ realises the minimum in~\eqref{KP} and $\Ups(f)= \mathcal{T}_c(f,g)$. As there is no duality gap in~\eqref{PPXcomp} nor in~\eqref{KP}, $\Ups^*(f) = \mathcal{T}_c^*(f,g)$. Additionally, $(\psi^{\bar c}, \psi^{\bar cc})$ is admissible for $\mathcal{T}_c^*(f,g)$ and $\psi = (\psi^{\bar cc})_-$. Thus
		\begin{equation}\label{proof_psi_KP}
			K_f(\psi^{\bar c}, \psi)=\int f\psi^{\bar c}\, d\lambda + \int (1-f)(\psi^{\bar cc})_-\, d\lambda = \mathcal{T}_c^*(f,g) \geq  \int f \psi^{\bar c}\, d\lambda + \int  g \psi^{\bar cc}\, d\lambda.
		\end{equation}
		
		\noindent
		Hence
		\[
		\int (1-f-g) (\psi^{\bar cc})_-\, d\lambda \ge \int  g (\psi^{\bar cc})_+\,d\lambda.
		\]
		
		\noindent
		Since $(1-f-g) (\psi^{\bar cc})_-\le0$ and $g (\psi^{\bar cc})_+ \geq 0$, the integrands must vanish $\lambda$-almost everywhere: we deduce~\eqref{psi_KP}. Additionally, the inequality in~\eqref{proof_psi_KP} is an equality. Consequently, $(\psi^{\bar c}, \psi^{\bar cc})$ is a pair of Kantorovitch potentials for~\eqref{DKP}.
	\end{proof}
	
	To end this section, we establish a comparison principle on the potentials maximising~\eqref{DPp}. We say that a set $\Psi \subset C(X)$ admits a minimal (respectively maximal) element for the relation $\leq$ if there exists $\psi_0 \in \Psi$ such that for any $\psi \in \Psi$, $\psi_0 \leq \psi$ (respectively $\psi_0 \geq \psi$).
	
	\begin{proposition}\label{prop_pot_monot}
		Assume that $X$ is a compact metric space and that $c \in C(X \times X, \R)$. Let $f \in L^1_m$ and let us define
		\[
		\Psi_f:=\lt\{\psi, \, \psi = \vhi^c_{\,-} \text{ for some } \vhi \in \Phi' \text{ and } K_f(\psi^{\bar c}, \psi)=\Ups^*(f)\rt\},
		\]
		
		\noindent
		where $K_f$ is defined in~\eqref{Kf} and $\Phi'$ in Lemma~\ref{lem_Phi}.
		
		\noindent
		Then:
		\begin{enumerate}[(i)]
			\item $\Psi_f$ admits a maximal element for the relation $\leq$, denoted by $\psi_f$ in the sequel,
			\item For $f_1, f_2 \in L^1_m$, there holds  $f_1 \leq f_2\implies\psi_{f_1}\geq \psi_{f_2}$.
		\end{enumerate}
	\end{proposition}
	\smallskip	
	
	\begin{proof}~\\
		\textit{Step 1. Sufficient condition and preliminary claim.}
		
		Notice that $\Psi_f$ is not empty by Proposition~\ref{prop_ex_DPp}. To obtain $(i)$, we prove that the set (which is not empty since $\Psi_f$ is not empty)
		\[
		\Phi_f := \lt\{\vhi \in \Phi', \, K_f(\vhi, \vhi^c_{\,-}) = \Ups^*(f)\rt\}
		\]
		
		\noindent
		admits a minimal element $\vhi_f$ and then $\psi_f := (\vhi^c_f)_-$ is the desired maximal element of $\Psi_f$. Let us make a preliminary observation.
		\begin{claim*}
			Let $f_1, f_2 \in L^1_m$ with $f_1 \leq f_2$ and set $\vhi_i \in \Phi_{f_i}$ for $i \in \{1, 2\}$. Then $\vhi_{\wedge} := \vhi_1 \wedge \vhi_2 \in \Phi_{f_1}$.
		\end{claim*}
		
		Let us first prove that $\vhi_{\wedge} \in \Phi'$. In the sequel we write $\vhi_i^c:=(\vhi_i)^c$ and ${\vhi_i^c}_-:=((\vhi_i)^c)_-$ for $i\in\{1,2,\wedge\}$. We observe that $\vhi_{\wedge} \in C(X)$. By definition of the $c$-transform, we obtain
		\begin{equation}\label{ctransfwedgegeq}
			\vhi_{\wedge}^c = (\vhi_1 \wedge \vhi_2)^c \geq \vhi_1^c \vee \vhi_2^c
		\end{equation}
		
		\noindent
		and
		\begin{equation}\label{ctransfwedgeq}
			(\vhi_1 \vee \vhi_2)^c = \vhi_1^c \wedge \vhi_2^c.
		\end{equation}
		
		\noindent
		Since for $i \in \{1, 2\}$, $\max \vhi_i^c \geq 0$ we have by~\eqref{ctransfwedgegeq} that $\max \vhi_{\wedge}^c \geq 0$. 
		
		We now prove that $({\vhi_{\wedge}^c}_-)^{\bar c} = \vhi_{\wedge}$. We observe that ${\vhi_{\wedge}^c}_- \leq \vhi_{\wedge}^c$ which implies $({\vhi_{\wedge}^c}_-)^{\bar c} \geq \vhi_{\wedge}^{c \bar c}$. By Proposition~\ref{prop_ctransf} $(ii)$, $\vhi_{\wedge}^{c \bar c} \geq \vhi_{\wedge}$ so that $({\vhi_{\wedge}^c}_-)^{\bar c} \geq \vhi_{\wedge}$. Conversely, taking the negative part of~\eqref{ctransfwedgegeq}, we have ${\vhi_{\wedge}^c}_- \geq (\vhi_1^c \vee \vhi_2^c)_- = {\vhi_1^c}_- \vee {\vhi_2^c}_-$. Taking the $\bar c$-transform and using~\eqref{ctransfwedgeq} (with $\bar c$ instead of $c$) yields
		\[
		({\vhi_{\wedge}^c}_-)^{\bar c} \leq ({\vhi_1^c}_- \vee {\vhi_2^c}_-)^{\bar c} = ({\vhi_1^c}_-)^{\bar c} \wedge ({\vhi_2^c}_-)^{\bar c} 
		= \vhi_1 \wedge \vhi_2 = \vhi_{\wedge}.
		\]
		
		\noindent
		Hence $({\vhi_{\wedge}^c}_-)^{\bar c} = \vhi_{\wedge}$ and $\vhi_{\wedge} \in \Phi'$. 
		
		We now show that the pair $(\vhi_{\wedge}, {\vhi_{\wedge}^c}_-)$ maximises $\Ups^*(f_1)$. We set 
		\[
		\Delta_K := K_{f_1}(\vhi_{\wedge}, {\vhi_{\wedge}^c}_-)-K_{f_1}(\vhi_1, {\vhi_1^c}_-) = \int f_1 (\vhi_{\wedge}-\vhi_1)\, + \int (1-f_1)({\vhi^c_{\wedge}}_--{\vhi_1^c}_-)\,.
		\] 
		
		\noindent
		By optimality of $\vhi_1$, $\Delta_K \leq 0$. Let us prove the converse inequality. Substituting $f_1 = f_2 + f_1-f_2$ in the definition of $\Delta_K$, we obtain
		\begin{equation*}
			\Delta_K=\int f_2 (\vhi_{\wedge} - \vhi_1) \,+\int (1-f_2)({\vhi_{\wedge}^c}_- - {\vhi_1^c}_-)\,
			+\int (f_2-f_1)(\vhi_1 - \vhi_{\wedge} + {\vhi_{\wedge}^c}_- - {\vhi_1^c}_-)\,.
		\end{equation*}
		
		\noindent
		We have $f_2-f_1\ge0$ and $\vhi_1- \vhi_{\wedge} \ge0$. Additionally, $\vhi_{\wedge}^c \geq \vhi_1^c$, so that ${\vhi_{\wedge}^c}_-\ge{\vhi_1^c}_-$. Thus the last integral in $\Delta_K$ is non-negative. Adding and subtracting $f_2 \vhi_2$ in the first integral yields
		\begin{equation}\label{lowerboundDeltaK}
			\Delta_K \geq \int f_2 \vhi_2 \,+ \int f_2 (\vhi_{\wedge}-\vhi_1-\vhi_2)\, +\int (1-f_2)({\vhi_{\wedge}^c}_--{\vhi_1^c}_-)\,.
		\end{equation}
		
		\noindent
		Let us set $\vhi_{\vee} := \vhi_1 \vee \vhi_2$. By optimality of $\vhi_2$, we have $K_{f_2}(\vhi_2,{\vhi_2^c}_-)\ge K_{f_2}(\vhi_{\vee}, {\vhi_{\vee}^c}_-)$, which rewrites as
		\[
		\int f_2\vhi_2\, \ge \int f_2 \vhi_{\vee}\,  + \int(1-f_2)({\vhi_{\vee}^c}_--{\vhi_2^c}_-)\,.
		\]
		
		\noindent
		Injecting this inequality in the first term of the right-hand side of~\eqref{lowerboundDeltaK} yields
		\begin{equation}\label{boundeddeltaK}
			\Delta_K \geq  \int f_2(\vhi_{\wedge} + \vhi_{\vee} -\vhi_1-\vhi_2)\,
			+\int (1-f_2)({\vhi_{\vee}^c}_-+{\vhi_{\wedge}^c}_--{\vhi_1^c}_--{\vhi_2^c}_-)\,.
		\end{equation}
		
		\noindent
		The integrand in the first integral of~\eqref{boundeddeltaK} vanishes. Regarding the second term, using~\eqref{ctransfwedgeq} and~\eqref{ctransfwedgegeq} we obtain
		\[
		{\vhi_{\vee}^c}_-+{\vhi_{\wedge}^c}_-\ge {\vhi_1^c}_-\wedge{\vhi_2^c}_-+ {\vhi_1^c}_-\vee{\vhi_2^c}_- ={\vhi_1^c}_-+{\vhi_2^c}_-.
		\]
		
		\noindent
		Hence the integrand in the second integral is non-negative. We conclude that $\Delta_K \geq 0$ and finally that $\Delta_K=0$ so that the claim is proved. \medskip
		
		\noindent
		\textit{Step 2. Construction of the minimal element of $\Phi_f$.}
		
		By Lemma~\ref{prop_ex_DPp}, $\Phi'$ is compact. As $\vhi \mapsto K_f(\vhi, \vhi^c_{\,-})$ is continuous for the norm of uniform convergence, $\Phi_f$ is compact as well. Let $(\vhi_j)_{j \geq 0}$ be a dense subset of $\Phi_f$. For $x \in X$ and $j \geq 0$, we define $\widetilde{\vhi}_j$ and $\vhi_f$ by
		\[
		\widetilde \vhi_j(x):=\min(\vhi_0(x),\dots,\vhi_j(x)) \qquad \text{and} \qquad \vhi_f(x):= \inf\{ \vhi(x), \, \vhi \in \Phi_f \}.
		\]
		
		\noindent
		Using our preliminary claim with $f_1=f_2=f$ recursively, we obtain that for any $j \geq 0$, $\widetilde \vhi_j \in \Phi_f$. As $\Phi_f$ is compact and $\widetilde \vhi_j\to\vhi_f$ pointwise, we obtain that  $\widetilde \vhi_j\to\vhi_f$ uniformly and $\vhi_f\in\Phi_f$, so that $\vhi_f$ is the desired minimal element of $\Phi_f$.
		
		\medskip
		
		\noindent
		\textit{Step 3. Conclusion.} 
		
		Taking $\psi_f:={\vhi_f^c}_-$ proves $(i)$. Let $f_1 \leq f_2$ as given in the statement of $(ii)$. By the previous step, there exist $\vhi_1, \vhi_2$ respective minimal elements for $\Phi_{f_1}$ and $\Phi_{f_2}$ such that $\psi_1 := {\vhi_1^c}_-$ and $\psi_2 :={\vhi_2^c}_-$ are respective maximal elements for $\Psi_{f_1}$ and $\Psi_{f_2}$. By the preliminary claim, $\vhi_1 \wedge \vhi_2 \in \Phi_{f_1}$ and by minimality of $\vhi_1$ we have $\vhi_1 \leq \vhi_1 \wedge \vhi_2$, so that $\vhi_2 \geq \vhi_1$. Hence $\psi_2 \leq \psi_1$.
	\end{proof}

	\section{Existence of maximisers of~\eqref{MPintro} for translation invariant costs in $\R^d$}\label{trslinv}
	
	We now assume that $X=\R^d$, that $\lambda$ is the Lebesgue measure and that $c(x,y)=k(y-x)$, with $k : \R^d \to \R_+$. We recall the following hypotheses on $k$.
	\begin{enumerate}[(H1)]
		\item$ k \in C(\R^d, \R_+)$, $k(0)=0$ and $k(x) \to \infty$ as $|x| \to \infty$,
		\item $\forall x \not = 0$,		
		\[
		\limsup_{r \to 0} \frac{1}{r^d} \lt|B_r(x) \cap \{y \in \R^d, \, k(y) < k(x)\} \rt| > 0,
		\]
		\item $\forall \, \sigma \in \mathbb{S}^{d-1}$, $r \mapsto k(r\sigma)$ is increasing on $\R_+$.	
	\end{enumerate}
	Notice that under hypotheses~(H\ref{cont})\&(H\ref{cone}), there holds $k(x)> 0$ for $x \neq 0$.
	\medskip
	
	The primal problem is now defined as
	\begin{equation}\label{PPRd}
		\Ups(f) :=\inf \lt\{\int c\,d\gamma : \gamma\in \PI_f\rt\},
	\end{equation}
	where
	\[
	\PI_f :=\lt\{\gamma\in\mathcal{M}_+(\R^d\times \R^d) : \gamma_x = f,\, \gamma_y \le 1-f \rt\}.
	\]
	
	The goal of this section is to prove that for  every  $m>0$ the energy
	\begin{equation}\label{MPintro}
		\E(m) := \sup \lt\{\Ups(f) : f \in L^1(\R^d, [0,1]), \int f \, = m \rt\}
	\end{equation}
	admits a maximiser.
	
	\subsection{First properties of $\Ups$ and saturation theorem.}\hfill
	
	In this subsection, we collect some properties of the functional $\Ups$ defined in $\R^d$ and establish a saturation property (Theorem~\ref{thm_saturation}), namely that if $\gamma$ is a minimiser for $\Ups(f)$ then $\gamma_y(x) \in \{f(x), 1-f(x)\}$ for almost every $x \in \R^d$.
	
	\medskip
	
	We start by proving that minimisers of \eqref{PPRd} exist. The proof of this result is similar to the proof of~\cite[Proposition 2.1]{CanGol}, but with weaker assumptions on the cost $c$ and in the context of functions taking values in $[0,1]$ rather than in $\{0, 1\}$.
	
	\begin{proposition}\label{prop_Ups_max}
		Assume that $k$ satisfies~(H\ref{cont}). Then, for any $m >0$ and $f \in L^1_m$, the infimum in~\eqref{PPRd} is attained. Additionally, given any minimiser $\gamma$ of~\eqref{PPRd} we have $\Ups(f) = \mathcal{T}_c(f,g)$, where $g := \gamma_y$.
		
		Lastly, there exists $R = R(m)$ non-decreasing in $m$ such that for any $f \in L^1_m$,
		\begin{equation}\label{infrestr}
			\Ups(f)=\min \lt\{\int c\, d\gamma :\gamma\in \PI_f, \, \forall \, (x,y) \in \supp \gamma, \, |x-y| \leq R \rt\}
		\end{equation}
		
		\noindent
		and for any minimiser $\gamma$ of~\eqref{PPRd}, there holds $|x-y| \leq R$ on $\supp\gamma$.
		
	\end{proposition}
	\smallskip
	
	\begin{proof}~\\
		The strategy of the proof is to first establish~\eqref{infrestr} with an infimum in place of the minimum. Then we use this property to derive compactness for~\eqref{PPRd}.
		
		\medskip
		
		\noindent
		\textit{Step 1. Restricting the set of competitors for~\eqref{PPRd}.}
		
		We let $\gamma \in \PI_f$ and set $g := \gamma_y$. We want to build a competitor $\widetilde{\gamma}$ for $\Ups(f)$ such that for some $R >0$, $|x-y| \leq R $ for every $ (x,y) \in \supp \widetilde{\gamma}$. For $R > 0$, we define
		\begin{equation*}
			\Gamma_R := \lt\{ (x,y) \in \R^d \times \R^d, \, |x-y| \geq R \rt\}.
		\end{equation*}
		
		\noindent
		We consider a standard partition of $\R^d$ into cubes $(Q_i)_{i \geq 0}$ with side-length $\rho_1(m) := (3m)^{1/d}$. We define
		\[
		I := \{ i \geq 0, \, m_i := \gamma(\Gamma_R \cap (Q_i \times \R^d))>0\},
		\]
		
		\noindent
		and for $ i \in I$, we set
		\[
		\gamma_{\text{bad}, i} := \chi_{\Gamma_R \cap (Q_i \times \R^d)} \gamma.
		\] 
		
		\noindent
		As $|Q_i| - \smallint f - \smallint g \geq m \geq m_i$, there exists a positive measure $\mu_i\ll\chi_{Q_i}\lambda$ such that 
		\[
		\mu_i \leq \chi_{Q_i}(1 - f - g) \qquad \text{and} \qquad \mu_i (\R^d)=\gamma_{\text{bad},i}(\R^d \times \R^d).\]
		
		\noindent
		Denoting by $\theta_i$ the first marginal of $\gamma_{\text{bad},i}$ we set
		
		\begin{equation*}
			\widetilde{\gamma}_i := \chi_{Q_i \times \R^d} \gamma - \gamma_{\text{bad},_i} + \frac{1}{m_i}\theta_i \otimes \mu_i\ \ge0.
		\end{equation*}
		\noindent
		For $i \in I^c$, we simply define $\widetilde{\gamma}_i := \chi_{(Q_i \times \R^d )} \gamma$. As a consequence, $\widetilde{\gamma} := \sum_{i \geq 0} \widetilde{\gamma_i}$ is a transport plan whose first marginal is $f$ and second marginal $\widetilde{g}$ verifies $\widetilde{g} \leq g \leq 1-f$. By construction, for $R > \sqrt{d}\rho_1(m)$, we have $\widetilde{\gamma}(\Gamma_R) = 0$. 
		
		Let us now compare the transportation cost of $\gamma$ and $\widetilde{\gamma}$. We compute:
		
		\begin{equation*}
			\begin{aligned}
				\int c \, d\widetilde{\gamma} -\int c \, d\gamma &= \sum_{i \in I} \int_{Q_i \times \R^d} c \, d\lt(\frac{\theta_i \otimes \mu_i}{m_i} - \gamma_{\mathrm{bad},i}  \rt) \\ 
				&\leq \lt(\sum_{i \in I} m_i \rt) \lt( \max_{z \in \overline{Q}_{\rho_1(m)}} k(z) - \inf_{|z|\geq R} k(z) \rt).
			\end{aligned}
		\end{equation*}
		
		\noindent
		Let us set $\overline{Q}_{\rho_1(m)}:=[0; \rho_1(m)]^d$ and then $M := \max\{k(z), \, z \in \overline{Q}_{\rho_1(m)}\}$. By~(H\ref{cont}), there exists $R> \sqrt{d}\rho_1(m)$ such that if $|z| > R$, then $k(z)>M$. With this choice of $R$ we have $\smallint c \,d\widetilde{\gamma} \leq \smallint c \,d\gamma$. Lastly, whenever $\gamma(\Gamma_R) >0$,
		\begin{equation}\label{stricineqgamma}
			\int c \,d\widetilde{\gamma} < \int c \,d\gamma.
		\end{equation}
		
		\noindent
		\textit{Step 2 : Lower semi-continuity of the transportation cost.}
		
		This step is classical. To prove that $\gamma \mapsto \smallint c \,d\gamma$ is lower semi-continuous with respect to weak convergence, we proceed by approximation. Let us assume that $\gamma_n \st*\rightharpoonup \gamma$ as $n \to \infty$. For $j \geq 0$, we define $c_j := c \wedge j$. The sequence $c_j$ is non-decreasing and converges pointwise to $c$. For every $j \geq 0$, $c_j \in C_b(\R^d\times \R^d)$, so that
		\[
		\int c_j \, d\gamma = \lim_n \int c_j \, d\gamma_n \leq \liminf_n\int c \, d\gamma_n.
		\] 
		
		\noindent
		By the monotone convergence theorem,
		\[
		\int c \, d\gamma = \lim_j \int c_j \, d\gamma \leq \liminf_n\int c \, d\gamma_n,
		\]
		
		\noindent
		which concludes the second step of the proof. \medskip
		
		\noindent
		\textit{Step 3. $\Ups(f)$ admits a minimiser.}
		
		Let $\gamma_n$ be a minimising sequence for $\eqref{PPRd}$. Let us show that the sequence $\gamma_n$ is tight. By the first step, we can assume that there exists $R = R(m)$ such that for any $n \geq 0$ there holds $|x-y| \leq R $ on $\supp \gamma_n$. Now, because $\smallint f \leq m < \infty$, there exists $R' = R'(m)> 0$ such that  $\smallint_{\R^d \setminus B_{R'}} f \leq \eps$. Hence,
		\begin{equation*}
			\gamma_n \lt(\R^d \times \R^d \setminus (B_{R'} \times B_{R+R'})\rt)= \gamma_n\lt(B_{R'} \times (\R^d \setminus B_{R+R'})\rt) + \gamma_n \lt((\R^d \setminus B_{R'}) \times \R^d\rt)  \leq 0 + \eps m
		\end{equation*}
		
		\noindent
		which proves that the sequence $\gamma_n$ is tight. Together with the second step, this shows that~\eqref{PPRd} admits a minimiser. Moreover, by~\eqref{stricineqgamma} for any minimiser $\gamma$ of~\eqref{PPRd} there holds $|x-y| \leq R$ on $\supp\gamma$. Lastly, setting  $g := \gamma_y$ the identity $\Ups(f) = \mathcal{T}_c(f,g)$ is immediate.
	\end{proof}
	
	We now establish some basic properties of the functional $\Ups$. The results here are similar to~\cite[Proposition 2.2 \& Lemma  2.4]{CanGol}.

	\begin{proposition}\label{prop_basic_ups}
		Assume that $k$ satisfies~(H\ref{cont}). Given $m>0$ and $f_1, f_2 \in L^1_m$ we have:
		
		\begin{enumerate}[(i)]
			\item If $f_1 + f_2 \leq 1$, then
			\begin{equation*}
				\Ups(f_1 + f_2) \geq \Ups(f_1) + \Ups(f_2).
			\end{equation*}
			
			\noindent
			As a consequence, if $f_1 \leq f_2$, then $\Ups(f_1) \leq \Ups(f_2)$.
			
			\item There exists $R=R(m)$ such that if $d(\supp f_1, \supp f_2) \geq R$, then
			\begin{equation*}
				\Ups(f_1 + f_2) = \Ups(f_1) + \Ups (f_2).
			\end{equation*}
			
			\item There exists $C=C(m)>0$ such that
			\[
			|\Ups(f_1)-\Ups(f_2)|\le C\|f_1-f_2\|_{L^1}.
			\]
			
			\item Let $f, f_n\in L^1(\R^d,[0,1])$ be such that the sequence $f_n$ is tight and $f_n\st*\rightharpoonup f$. Then $\Ups(f_n)\to\Ups(f)$.
		\end{enumerate}
		
	\end{proposition}
	\smallskip

	\begin{proof}~\\
		\textit{Step 1. Proof of (i)\&(ii)}.
		
		To prove $(i)$, we consider a transport plan $\gamma$ optimal for $\Ups(f_1 + f_2)$ whose existence is guaranteed by Proposition~\ref{prop_Ups_max}. We would like to extract from $\gamma$ two plans $\gamma^1$ and $\gamma^2$ admissible for $\Ups(f_1)$ and $\Ups(f_2)$ respectively. Using the convention $0/0 = 0$, we define $\gamma^1$ and $\gamma^2$ through
		\[
		d\gamma^1(x,y) :=  \frac{f_1(x)}{(f_1 + f_2)(x)} \, d\gamma(x,y) \qquad \text{and} \qquad
		d\gamma^2(x,y) :=  \frac{f_2(x)}{(f_1 + f_2)(x)} \, d\gamma(x,y).
		\]
		
		\noindent
		By construction, $\gamma^1_x = f_1$ and $\gamma^2_x = f_2$. We also have $\gamma^1 \leq \gamma$, so that 
		\[
		\gamma^1_y \leq \gamma_y \leq 1-(f_1+f_2) \leq 1-f_1.
		\]
		
		\noindent
		Likewise, $\gamma^2_y \leq 1 - f_2$. Therefore, $\gamma^1$ and $\gamma^2$ are admissible for $\Ups(f_1)$ and $\Ups(f_2)$ respectively. Moreover
		\begin{equation*}
			\Ups(f_1) + \Ups(f_2) \leq \int c \, d\gamma^1 + \int c\, d\gamma^2 = \int c \, d\gamma = \Ups(f_1 + f_2),
		\end{equation*}
		
		\noindent
		which is the desired conclusion.
		
		To prove $(ii)$, we consider transport plans $\gamma^1$ and $\gamma^2$ which are optimal for $\Ups(f_1)$ and $\Ups(f_2)$ respectively. We define $g_1 := \gamma^1_y$ and $g_2 := \gamma^2_y$. If we set $\gamma := \gamma^1 + \gamma^2$, we have $\gamma_x = f_1+f_2$ and $\gamma_y = g_1 + g_2$. Moreover, by Proposition~\ref{prop_Ups_max}, if $d(\supp f_1, \supp f_2) \geq R$ for $R=R(m)$ large enough, then the supports of $g_1$ and $g_2$ are also disjoint. Consequently, $g_1 + g_2 \leq 1 - (f_1 + f_2)$, so that $\gamma$ is admissible for $\Ups(f_1+f_2)$ and we have the desired converse inequality
		\begin{equation*}
			\Ups(f_1 + f_2) \leq \int c \, d(\gamma^1 + \gamma^2) \leq \Ups(f_1) + \Ups(f_2).
		\end{equation*}
		
		\noindent	
		\textit{Step 2. Proof of (iii).}
		
		Exchanging the roles of $f_1$ and $f_2$, it is enough to prove the estimate
		\begin{equation}\label{upperLip}
			\Ups(f_2)-\Ups(f_1) \leq C \|f_2-f_1\|_{L^1}.
		\end{equation}
		Let $\gamma^1$ be a minimiser of $\Ups(f_1)$ and let us set $g_1 := \gamma^1_x$. In the next substeps, we build from $\gamma^1$ an exterior transport plan $\gamma^2$ for $f_2$ with controlled cost.
		\medskip
		
		\noindent
		\textit{Step 2.a. Transporting most of $f_1 \wedge f_2$.}
		
		Using the convention $0/0= 0$, we define a plan $\gamma'$ by
		\[
		d\gamma'(x,y) :=  \frac{(f_1 \wedge f_2)(x)}{f_1(x)} d\gamma_1(x,y).
		\]
		
		\noindent
		We set $g' := \gamma'_y$. Notice that $\gamma' \leq \gamma^1$, which implies $g' \leq g_1$. Additionally, $\gamma'_x = f_1 \wedge f_2$, so that
		\begin{equation}\label{massg2a}
			\gamma^1(\R^d \times \R^d) - 	\gamma'(\R^d \times \R^d) = \int(f_1-f_1\wedge f_2)\,= \int (f_1 - f_2)_+\,.
		\end{equation}
		Heuristically, $\gamma'$ corresponds to sending through $\gamma^1$ as much mass from $f_2$ as possible. However, we have to remove some of this mass because the constraint $g' \leq 1 - f_2$ might not hold true everywhere. Let
		\[
		u:=(f_2+g'-1)_+
		\]
		and define $\gamma''$ as
		\[
		d\gamma''(x,y) :=  \frac{g'(y)-u(y)}{g'(y)} \, d\gamma'(x,y).
		\]
		
		\noindent
		We set $f'' := \gamma''_x$ and $g'':=\gamma''_y$. By construction, $g''=g'-u$ so $g''\le1-f_2$ as desired. Since  $\gamma'' \leq \gamma'$, we also have $f'' \leq f_1 \wedge f_2 \leq f_2$. Now since $g' \leq g_1 \leq 1 -f_1$ we have $u\le (f_2-f_1)_+$ from which we infer
		
		\begin{equation*}
			\gamma'(\R^d \times \R^d) - \gamma''(\R^d \times \R^d) =\int (g'-g'')\,=\int u\, \leq \int (f_2 - f_1)_+\,.
		\end{equation*}
		
		\noindent
		Summing this and~\eqref{massg2a} yields
		\begin{equation}\label{massg2b}
			\gamma^1(\R^d \times \R^d) - \gamma''(\R^d \times \R^d) \leq \|f_2-f_1\|_{L^1}.
		\end{equation}
		
		\noindent
		Eventually since $\gamma''\le\gamma'\le\gamma^1$ and $c\ge0$ we have
		\begin{equation}\label{boundedg2b}
			\int c \, d\gamma'' - \int c \, d\gamma^1 \le0.
		\end{equation}
	
		\noindent
		\textit{Step 2.b. Final construction.}
		
		We are now ready to build an admissible transport plan $\gamma^2$ for $\Ups(f_2)$. Noticing that $f_2 - f''\ge0$ we write $f_2 = f'' + (f_2 - f'')$. By~\eqref{massg2b} we have
		\begin{equation}\label{totalmassr}
			\int (f_2 - f'')\, = \int(f_2- f_1)\, + \int(f_1 -f'')\, \leq 2 \|f_2-f_1\|_{L^1}.
		\end{equation}
		
		\noindent
		Arguing as in the proof of Proposition~\ref{prop_Ups_max}, we can find  a function $0\le g''' \leq 1- f_2 - g''$ with $\int g'''\,\le\int(f_2-f'')\,$ and a transport plan $\gamma'''$ between $f_2-f''$ and $g'''$ such that for some $C=C(m)>0$,
		\begin{equation}\label{boundedgr}
			\int c \, d\gamma''' \leq C\int(f_2 - f'')\, \st{\eqref{totalmassr}}{\leq} C \|f_2 - f_1\|_{L^1}.
		\end{equation}
		
		\noindent
		Finally we define $\gamma^2 := \gamma'' + \gamma'''$ which is admissible for $\Ups(f_2)$ by construction. Summing~\eqref{boundedg2b} and~\eqref{boundedgr}, we get 
		\[
		\Ups(f_2)\le\int c\,d\gamma^2\le\int c\, d\gamma_1+C \|f_2 - f_1\|_{L^1}=\Ups(f_1)+C \|f_2 - f_1\|_{L^1}.
		\]
		This proves~\eqref{upperLip}  and thus point~(iii).
		
		\medskip
		
		\noindent
		\textit{Step 3. Proof of (iv).}
		
		Let $f_n$ and $f$ be as in the statement of the proposition. By weak convergence, we have $f_n,f\in L^1_m$ for some $m>0$. Using the Lipschitz continuity of $\Ups$ with respect to $L^1$ convergence, we may assume without loss of generality that $f_n$ (and thus also $f$) are supported in $\overline{B}_{R_0}$ for some $R_0>0$. Applying Proposition~\ref{prop_Ups_max} we get that minimisers of $\Ups(f_n)$ and $\Ups(f)$ are supported in $\overline{B}_R\times\overline{B}_R$ for some $R>R_0>0$. We may thus restrict these problems to the compact set $\overline{B}_R$. Using Proposition~\ref{prop_DP=OP} we have $\Ups(f_n)=\Ups^*(f_n)$ and it is thus enough to prove the continuity of $\Ups^*$ with respect to the weak-$*$ topology.\\
		By Proposition~\ref{prop_ex_DPp}, for every $n\ge 0$  there exists a pair of potentials $(\vhi_n, \psi_n)$ maximising $\Ups^*(f_n)$. Since for every $n$, $\vhi_n$ belongs to $\Phi'$ (where $\Phi'$ is defined by \eqref{defPhiprime}) and since this set is compact by Proposition~\ref{prop_ex_DPp} we have that a subsequence $\vhi_{n'}$ of $\vhi_n$ converges in $C(\overline{B}_R)$ to some $\vhi\in \Phi'$. Arguing as in the proof of Proposition~\ref{prop_ex_DPp} we see that $\psi_{n'}$ also converges to $\psi$ with $(\vhi,\psi)$ admissible for $\Ups^*(f)$. By weak-strong convergence we then have
		\[
		\limsup\Ups^*(f_{n'})= \limsup \int f_{n'} \vhi_{n'} +(1-f_{n'}) \psi_{n'} \,=\int f \vhi +(1-f)\psi\,\le \Ups^*(f).
		\]
		Similarly, if $(\vhi,\psi)$ are optimal potentials for $\Ups^*(f)$, they are admissible for $\Ups^*(f_n)$ and thus
		\[
		\liminf\Ups^*(f_n)\ge \liminf
		\int f_n \vhi +(1-f_n) \psi =\int f \vhi +(1-f)\psi= \Ups^*(f).
		\]
		We then have $\lim \Ups^*(f_{n'})=\Ups^*(f)$ and by uniqueness of the limit we see that the extraction was not necessary. This establishes~(iv) and ends the proof of the proposition.
	\end{proof}
	
	The next lemma and theorem state very important saturation properties satisfied by the optimal exterior transport plan. These results extend~\cite[Lemma 5.1 \& Proposition 5.2]{DePMSV} to more general costs $c$.
	
	\begin{lemma}\label{lem_saturation}
		Assume that $k$ satisfies~(H\ref{cont})\&(H\ref{cone}). For  $f \in L^1_m$ let $\gamma$ be optimal for $\Ups(f)$. Then for every $(x_0,y_0)\in\supp\gamma$ there holds $f+\gamma_y\equiv1$ almost everywhere on the saturation set 
		\[
		S(x_0, y_0) := \{ y \in \R^d, \, k(y-x_0) < k(y_0-x_0) \}.
		\]
	\end{lemma}
	
	\begin{proof}~\\
		In the proof we set $g := \gamma_y$ and $h := f+g$. Let $(x_0, y_0)\in\supp\gamma$ and assume without loss of generality that $x_0 = 0$. We suppose by contradiction that there exists $\eps > 0$ such that the set
		\[S_{\eps} := \{h < 1\} \cap \{y \in \R^d, \, k(y) < k(y_0) - \eps  \}\]
		
		\noindent
		has positive Lebesgue measure. Notice that by~(H\ref{cont}), $k(x) \to \infty$ as $|x|\to \infty$ so that $S_{\eps}$ is bounded. Therefore, 
		\[
		m_{\eps} := \int_{S_{\eps}} (1-h)\,\ \in (0, \infty).
		\] 
		
		\noindent
		We now exhibit an exterior transport plan $\widetilde{\gamma}$ whose transportation cost is strictly smaller than the one of $\gamma$. Given $r >0$, we define the measure $\gamma^0 := \gamma \restr (B_r \times B_r(y_0))$. As $(0, y_0) \in \supp \gamma$, for every $r> 0$,
		\begin{equation}\label{defsup}
			0 < \gamma^0(\R^d \times \R^d) \leq \int_{B_r} f\, \leq |B_r|.
		\end{equation}
		
		\noindent
		Thus, by the last inequality in~\eqref{defsup} there exists $r_{\eps} > 0$ such that for every $r \in (0, r_\eps]$,
		\begin{equation*}
			\gamma^0(\R^d \times \R^d) = \alpha m_{\eps}
		\end{equation*}
		
		\noindent
		for some $0 < \alpha \leq 1$. Let us fix $r \in (0, r_{\eps}]$. We define a competitor $\widetilde\gamma$ for $\Ups(f)$ by setting $\widetilde \gamma := \gamma - \gamma^0 + \eta$, where
		\begin{equation*}
			\eta := \gamma_x^0 \otimes \frac{1-h}{m_{\eps}}\chi_{S_{\eps}}.
		\end{equation*}
		
		\noindent
		By construction, $\widetilde\gamma_x = \gamma_x = f$. We also have
		\begin{equation*}
			f+ \widetilde\gamma_y \leq f +g + \alpha(1-h)\chi_{S_{\eps}} \leq h + \alpha(1-h) = 1 - (1-h)(1-\alpha) \leq 1,
		\end{equation*}
		
		\noindent
		so that $\widetilde\gamma$ is admissible for $\Ups(f)$. We compute
		\[
		\int c \, d\widetilde\gamma - \int c \, d\gamma = \int c \, d\eta - \int c \, d\gamma^0
		\leq \alpha m_{\eps} \lt(\max_{\overline{B}_r \times  \overline{S}_{\eps}} c(x,y) - \min_{\overline{B}_r \times \overline{B}_r(y_0)} c(x,y)\rt).
		\]
		
		\noindent
		By continuity of $c$ there exists $r_\eps>0$ such that for  $0<r\le r_{\eps}$,
		\[ 
		\max_{\overline{B}_r \times  \overline{S}_{\eps}} c(x,y) \leq k( y_0) - \eps/2 \qquad \text{and} \qquad
		\min_{\overline{B}_r \times \overline{B}_r(y_0)} c(x,y) \geq k( y_0) - \eps/4.
		\]
		
		\noindent
		Thus for $0<r \leq r_{\eps}$,
		\[
		\int c \, d\widetilde\gamma - \int c \, d\gamma \leq -\alpha \eps m_{\eps}/4 <0,
		\]
		
		\noindent
		which contradicts the fact that $\gamma$ is a minimiser for $\Ups(f)$.		
	\end{proof}	
	
	\begin{theorem}
		\label{thm_saturation}
		Assume that $k$ satisfies~(H\ref{cont})\&(H\ref{cone}). For  $f \in L^1_m$, let $\gamma\in\PI_f$ be a minimiser of~\eqref{PPRd} and set $g := \gamma_y$. Then, defining
		\[
		E:=\{x : \exists\,  y \neq x\text{ such that } (x,y) \in \supp \gamma \text{ or } (y,x) \in \supp \gamma\},
		\]
		the set $E$ is Lebesgue measurable and we have the identity $g = (1-f)\chi_E + f\chi_{E^c}$.
	\end{theorem}
	\smallskip
	
	\begin{proof}~\\
		\textit{Step 1. A preliminary claim.}
		
		We first prove the following. Let $\mu, \nu \in \mathcal{M}_+(\R^d)$ be such that $\mu(\R^d) = \nu(\R^d)$, and let $\gamma \in \PI(\mu, \nu)$.  If we define the set
		\[
		\mathcal{A}(\gamma) := \{ x : \exists \, y \neq x\text{ such that }(x,y) \in \supp \gamma \},
		\]
		
		\noindent
		then $\mu \leq \nu$ on $\mathcal{A}(\gamma)^c$. 
		
		\medskip
		
		To prove the claim, let us first show that $\mathcal{A}(\gamma)$ is  $\mu$-measurable. We define
		\[
		\mathcal{D}(\gamma) := \supp \gamma \setminus \{(x,x): \, x \in \R^d\},
		\]
		
		\noindent
		which is a Borel set of $\R^d \times \R^d$. If we denote by $p_x : X \times X \to X$ the canonical projection on the first variable, we have
		\[
		p_x(\mathcal{D}(\gamma)) = \{x \in \R^d: \, \exists y \not = x, \, (x,y) \in \supp \gamma \} = \mathcal{A}(\gamma).
		\]
		
		\noindent
		Thus $\mathcal{A}(\gamma)$ is the image of a Borel set by a continuous map. By \cite[Proposition 2.2.13]{fed},  it is therefore $\mu-$measurable.
		
		We now show that $\mu \le \nu$ on $\mathcal{A}(\gamma)^c$. Let $\phi \in C_c(\R^d, \R_+)$. By definition of $\mathcal{A}(\gamma)$, if $(x,y) \in \supp \gamma$ and $x\in\mathcal{A}(\gamma)^c$ then $x=y$. Therefore
		\begin{align*}
			\int_{\mathcal{A}(\gamma)^c} \phi \, d\mu = \int  \phi(x) \chi_{\mathcal{A}(\gamma)^c}(x) \, d\gamma(x,y)
			&= \int \phi(y) \chi_{\mathcal{A}(\gamma)^c}(y) \chi_{\mathcal{A}(\gamma)^c}(x) \, d\gamma(x,y)\\
			&\le \int \phi(y) \chi_{\mathcal{A}(\gamma)^c}(y) \, d\gamma(x,y) =\int_{\mathcal{A}(\gamma)^c}\phi \, d\nu,
		\end{align*}
		and the claim is proved.
		
		\medskip
		
		\noindent
		\textit{Step 2. Construction of $E$.}
		
		We now consider an optimal exterior transport plan $\gamma$ for $\Ups(f)$ and set $g := \gamma_y$, $h := f+g$. By Proposition~\ref{prop_Ups_max}, $\gamma$ is an optimal transport plan from $f$ to $g$. Let $\overline\gamma$ be the image of $\gamma$ through the map $(x,y) \mapsto (y,x)$ and define
		\[
		E := \mathcal{A}(\gamma) \cup \mathcal{A}(\overline\gamma).
		\]
		
		\noindent
		We have $E^c=\mathcal{A}(\gamma)^c\cap\mathcal{A}(\overline\gamma)^c$ and by the first step there holds $f\le g$ and $g\le f$ almost everywhere on $E^c$. Hence, 
		\[
		g\chi_{E^c}=f\chi_{E^c}.
		\]
		To conclude the proof, we have to show that $g \equiv 1-f$ on $E$ or equivalently that up to Lebesgue negligible sets $\mathcal{A}(\gamma)$ and $\mathcal{A}(\overline\gamma)$ are included in $\{g = 1-f\}$. 
		
		On the one hand, if $x_0 \in \mathcal{A}(\gamma)$ is a Lebesgue point of both $f$ and $g$, there exists $y_0 \neq x_0$ such that $(x_0, y_0) \in \supp \gamma$. By Lemma~\ref{lem_saturation}, denoting
		\[
		S(x_0, y_0) := \{y \in \R^d, \, k(y-x_0) < k(y_0-x_0)\},
		\]
		we have $g = 1-f$ almost everywhere on $S(x_0, y_0)$. Notice that $S(x_0, y_0)$ is an open set and that $x_0 \in S(x_0, y_0)$ (since for $x \neq 0$, $k(x) >0=k(0)$). Hence $g(x_0) = 1-f(x_0)$ and $\mathcal{A}(\gamma) \subset \{g = 1-f\}$ up to a set of Lebesgue measure zero	.
		
		On the other hand, if $y_0 \in \mathcal{A} (\overline\gamma)$ there exists $x_0 \neq y_0$ such that $(x_0, y_0) \in \supp \gamma$. Let us assume by contradiction that $g(y_0) < 1 - f(y_0)$. Without loss of generality, we can assume that $y_0$ is a point of Lebesgue density one of $ \{g < 1 -f\}$.  Then
		\begin{equation*}
			\lim_{r \to 0} \frac{1}{r^d}\Big|\lt\{g = 1-f\rt\} \cap B(y_0, r) \Big| = 0.
		\end{equation*}
		
		\noindent
		Thus by Lemma~\ref{lem_saturation},
		\begin{equation*}
			\lim_{r \to 0} \frac{1}{ r^d}\lt|S(x_0, y_0) \cap B(y_0, r) \rt| = 0,
		\end{equation*}
		
		\noindent
		which contradicts~(H\ref{cone}) as  $y_0 \neq x_0$. Hence $g(y_0) = 1-f(y_0)$ and $\mathcal{A}(\overline\gamma) \subset \{g = 1-f \}$. This concludes the proof of the theorem.
	\end{proof}
	
	An important corollary is the uniqueness of the second marginal of minimisers of~\eqref{PPRd}.
	
	\begin{corollary}\label{coro_unique_g}
		Assume that $k$ satisfies~(H\ref{cont})\&(H\ref{cone}). Let $f \in L^1(\R^d)$. Then all minimisers $\gamma$ of $\Ups(f)$ have the same second marginal $\gamma_y$.
	\end{corollary}
	\smallskip
	
	\begin{proof}~\\
		We let $\gamma, \gamma'$ be two minimisers of~\eqref{PPRd} and define $\widetilde \gamma := (\gamma + \gamma')/2$ which also minimises~\eqref{PPRd}. We denote $g := \gamma_y$, $g' := \gamma'_y$ and $\widetilde{g} := \widetilde{\gamma}_y$ and introduce the set
		\[
		F := \{x \in \R^d: g(x), g'(x), \widetilde{g}(x) \in \{f(x), 1-f(x)\}\}.
		\]
		Assuming by contradiction that $g(x) \not = g'(x)$ for some $x \in F$ , we have $1/2 = \widetilde{g}(x) \in \{f(x), 1-f(x)\}$ so that $f(x) = 1/2$ and $g(x) = g'(x) = 1/2$, which is absurd. Hence $g = g'$ on $F$ and since $F$ is of full measure by Theorem~\ref{thm_saturation}, the proof is complete.
	\end{proof}
	
	\subsection{Preliminary results for the existence of a maximiser of~\eqref{MP}.}\hfill
	
	We now gather results which, combined with Theorem~\ref{thm_saturation}, allow us to prove existence of a maximiser for both \eqref{maxE} and \eqref{MP}.\\
	
	We first establish a corollary of Theorem~\ref{thm_saturation} regarding the monotonicity of the sum of the marginals of solutions to~\eqref{PPRd}.
	
	\begin{corollary}
		\label{coro_monot_f}
		Assume that $k$ satisfies~(H\ref{cont})\&(H\ref{cone}). Let $m >0$, let $f_1,f_2 \in L^1_m$ be such that $f_1 \leq f_2$ and let $\gamma^1, \gamma^2$ be respective minimisers of $\Ups(f_1)$ and $\Ups(f_2)$. Then setting $g_1 := \gamma^1_y $ and $g_2 := \gamma^2_y$, we have 
		$f_1+ g_1 \le f_2+ g_2.$
	\end{corollary}
	\smallskip	
	
	\begin{proof}~\\	
		Let $f_1, f_2 \in L^1_m$ be such that $f_1 \leq f_2$. In the first three  steps of the proof, we additionally assume that they are compactly supported. This condition is relaxed in the fourth and final step.\\
		By Proposition~\ref{prop_Ups_max}, we can assume that the ambient space is a compact ball $\overline{B}_R$. Let $\gamma^1, \gamma^2$ be minimisers for $\Ups(f_1)$ and $\Ups(f_2)$ respectively. For $i \in \{1,2\}$ we define $g_i:= \gamma^i_y$, $h_i:=f_i+g_i$ and set
		\[
		F:=\{h_1 > h_2\}.
		\]
		
		\noindent
		We shall prove that $|F|=0$. By Theorem~\ref{thm_saturation}, there exists $E_1, E_2 \subset \overline{B}_R$ such that
		\begin{equation*}
			h_1 = \chi_{E_1} + 2 f_1 \chi_{E_1^c} \qquad \text{and} \qquad h_2 = \chi_{E_2} + 2f_2 \chi_{E_2^c}.
		\end{equation*}
		
		\noindent
		Since $h_2\ge0$,  $h_1\le 1$ and  $h_2\ge f_2 \geq f_1$ we have
		\begin{equation}\label{honF}
			h_1>0,\qquad h_2=2f_2<1\qquad\text{and}\qquad f_1 < 1\quad  \text{on } F.
		\end{equation}
		
		\noindent
		\textit{Step 1. $|E_1^c \cap F|=0$.}
		
		By definition of $E_1$ we have $h_1= 2f_1$ on $E_1^c$ and by~\eqref{honF} we have $h_2 = 2f_2$  on $F$ and since $f_1\le f_2$ we get $h_1\le h_2$ on $E_1^c \cap F$. This contradicts the definition of $F$, hence $E_1^c \cap F  =\emptyset$ and in particular $|E_1^c \cap F | = 0$. Notice that as a consequence $h_1=1$ on $F$.
		\medskip
		
		\noindent
		\textit{Step 2. Intermediate claim.}
		
		Let $\psi_1$ be the maximal potential for $\Ups^*(f_1)$ given by Proposition~\ref{prop_pot_monot}. We define 
		\[
		G := \{\psi_1^{\bar cc} =0\}\cap E_1 \cap F
		\] 
		
		\noindent
		and claim that $|G| = 0$. Let us assume by contradiction that $|G|>0$. First notice that on $F$, 
		\[
		f_1 +g_1 =h_1> h_2 = 2f_2,
		\]
		so that 
		\[
		g_1 > 2f_2 - f_1\geq f_2 \geq f_1.
		\]
		Thus 
		\begin{equation}\label{Gsubg1f1}
			G\subset E_1 \cap F \subset \{ g_1 > f_1 \}.
		\end{equation}
		Now recall that by Theorem~\ref{thm_saturation},
		\[
		E_1=\{x : \exists y \neq x\text{ such that } (x,y) \in\supp\gamma^1\text{ or }(y,x)\in\supp\gamma^1\}.
		\]
		
		\noindent
		Together with~\eqref{Gsubg1f1} we obtain that for almost every $y_0 \in G$ there exists $x_0 \neq y_0$ with $(x_0, y_0) \in \supp \gamma^1$. Without loss of generality, we assume that $y_0$ is a point of positive density of $G$ and we set
		\[
		S(x_0, y_0) := \lt\{y \in \R^d: \, k(y-x_0) < k(y_0 - x_0) \rt\}.
		\]
		
		\noindent
		By~(H\ref{cone}), we have $|G \cap S(x_0, y_0)|> 0$. Let now $\widetilde{y} \in G \cap S(x_0, y_0)$. By Proposition~\ref{prop_Ups_OT}, $(\psi^{\bar c}_1, \psi^{\bar cc}_1)$ forms a pair of Kantorovitch potentials for the optimal transport from $f_1$ to $g_1$. Thus
		\[
		\psi^{\bar c}_1 (x_0) + \psi^{\bar cc}_1(y_0) = k(y_0-x_0) \qquad \text{and} \qquad \psi^{\bar c }_1(x_0) + \psi^{\bar cc}_1(\widetilde{y}) \leq k(\widetilde{y}-x_0).
		\]
		
		\noindent
		However, $y_0, \widetilde{y} \in G$, so that $\psi^{\bar cc}_1(y_0) = \psi^{\bar cc}_1(\widetilde{y}) = 0$, hence
		\[
		\psi^{\bar c}_1 (x_0) = k(y_0-x_0) \qquad \text{and} \qquad \psi^{\bar c }_1(x_0) \leq k(\widetilde{y}-x_0).
		\] 
		Eventually, as $\widetilde{y} \in S(x_0, y_0)$, we conclude that
		\[
		\psi^{\bar c}_1 (x_0) \leq k(\widetilde{y}-x_0) < k(y_0-x_0) = \psi^{\bar c}_1 (x_0),
		\]		
		\noindent
		obtaining a contradiction. Thus $|G|=0$, which is the claim.
		
		\medskip
		
		\noindent
		\textit{Step 3. $|E_1 \cap F|=0$.}
		
		By Proposition~\ref{prop_Ups_OT}, 
		\[
		\{\psi^{\bar cc}_1 > 0\} \subset \{g_1 = 0\} \quad\text{ so that }\quad \{g_1 > 0\} \subset \{\psi^{\bar cc}_1 \leq 0\}.
		\]
		We observe that $g_1 = 1-f_1$ on $E_1 \cap F$. By~\eqref{honF}, $E_1 \cap F \subset \{ g_1 > 0\}$ and by the previous step, $E_1 \cap F \subset \{\psi^{\bar cc}_1 \neq 0\}$, hence $\psi_1^{\bar cc} < 0$ almost everywhere on $E_1 \cap F$. 
		
		Let $\psi_2$ be the maximal potential for $\Ups^*(f_2)$ given by Proposition~\ref{prop_pot_monot}. As $f_1 \leq f_2$, we have $\psi_1 \geq \psi_2$ so that $\psi^{\bar cc}_1 \geq \psi^{\bar cc}_2$. Thus 
		\[
		\psi^{\bar cc}_2 < 0 \text{ on } E_1 \cap F.
		\]
		By Proposition~\ref{prop_Ups_OT} we deduce that 
		\[
		h_2 = g_2 + f_2 = 1 \text{ on } E_1 \cap F.
		\] 
		But since $h_2 < 1$ on $F$ we get that $|E_1 \cap F | = 0$ and with the first step we conclude that $|F| = 0$.
		
		\medskip
		
		\noindent
		\textit{Step 4. Extension to the non-compact case}.
		
		Let $f_1,f_2 \in L^1_m$ be such that $f_1 \leq f_2$. For $i \in \{1, 2\}$, we set $f_{i, R} = f_i \chi_{B_R}$, consider $\gamma^i_R$ an optimal exterior transport plan for $\Ups(f_{i, R})$ and set $g_{i, R} := (\gamma^i_R)_y$. Applying the previous steps to $f_{1,R}$ and $f_{2, R}$, we obtain
		\begin{equation}\label{compactineq}
			f_{1,R} + g_{1,R} \leq f_{2,R} + g_{2,R}.
		\end{equation}
		
		\noindent
		For $i \in\{1, 2\}$, $f_{i,R}$ $L^1$-converges to $f_i$ as $R \to \infty$. By Proposition~\ref{prop_basic_ups} $(iii)$,  $\Ups(f_{i, R}) \to \Ups(f_i)$ as $R \to \infty$. Additionally, $\gamma^i_R$ admits a subsequence converging weakly-$*$ to some $\widetilde{\gamma}^i$ admissible for $\Ups(f_i)$. By lower semi-continuity of $\gamma \mapsto \smallint c \, d \gamma$ with respect to weak-$*$ convergence, we get
		\[
		\int c \, d\widetilde{\gamma^i} \leq \liminf_R \int c \, d\gamma^i_R = \liminf_R \Ups(f_{i,R}) = \Ups(f_i).		
		\]
		
		\noindent
		Hence $\widetilde{\gamma}^i$ is optimal for $\Ups(f_i)$, so that by Corollary~\ref{coro_unique_g}, $\widetilde{\gamma}^i_y = g_i$. Finally, as $\gamma^i_R\st*\rightharpoonup \widetilde{\gamma}^i$ as $R \to \infty$, $g_{i,R}$ converges in duality with $C_b(\R^d)$ to $g_i$ as $R \to \infty$. Multiplying~\eqref{compactineq} by $\phi \in C_c(\R^d, \R_+)$, integrating and passing to the limit we obtain that for any $\phi \in C_c(\R^d, \R_+)$,
		\[
		\int (f_1 + g_1) \phi\, \leq \int (f_2+g_2) \phi\,.
		\]
		
		\noindent
		Hence $f_1 + g_1 \leq f_2 + g_2$ which completes the proof.
	\end{proof}
	
	We now prove that $\E$ is strictly superadditive.	
	
	\begin{proposition}\label{prop_mu_inc} Assume that $k$ satisfies~(H\ref{cont})\&(H\ref{monot}). Let $m \in (0, \infty)$ and define $e(m) :=\E(m)/m$. Then, $e$ is increasing on $(0, \infty)$. As a consequence, given $0 < m' < m$,
		\[
		\E(m') + \E(m-m') < \E(m).
		\]
	\end{proposition}
	
	\begin{proof}~\\		
		Let $M>m>0$. We have to establish that $\E(m) < (m/M)\E(M)$. \medskip
		
		\noindent
		\textit{Step 1. $\E(m) \leq (m/M)\E(M)$. }
		
		\noindent
		For $R>0$ we set
		\[
		\Gamma_R:=\{(x,y) \in \R^d \times \R^d: |x-y|> R\}.
		\]
		
		\noindent
		Let $0 \leq \eps < \E(m)/2$ and $f \in L^1_m$ of mass exactly $m$ and such that $\Ups(f)\ge \E(m)-\eps$. We denote $\lambda:=(M/m)^{1/d} >1$ and we set 
		\[
		f_{\lambda}(x):=f(x/\lambda)\qquad\text{for }x\in\R^d,
		\]
		
		\noindent
		so that $\int f_{\lambda}\,=M$. Let $\gamma^{\lambda}$ be an optimal transport plan for $\Ups(f_{\lambda})$. We define a Radon measure $ \gamma$ by
		\[
		\int \xi(x,y)\, d\gamma(x,y) := \frac{m}{M}\int \xi(x/\lambda,y/\lambda) \,d\gamma^{\lambda}(x,y)\qquad\text{for }\xi\in C_c(\R^d \times \R^d).
		\]
		\noindent
		Observe that $\gamma$ is admissible for $\Ups(f)$.
		By Proposition~\ref{prop_Ups_max}, there exists $R_{\lambda} = R_{\lambda}(M)$ such that $\gamma^{\lambda}(\Gamma_{R_{\lambda}})=0$.  Setting $R := R_{\lambda}/\lambda$, we then have
		\begin{equation}\label{defR}
			\gamma(\Gamma_R) = \frac{m}{M} \gamma^{\lambda}(\Gamma_{R_{\lambda}}) = 0.
		\end{equation}
		
		\noindent
		Let us define
		\[
		\kappa(r):=\min \lt\{ k(z)-k(z/\lambda) :r \leq |z|\leq R_{\lambda} \rt\}.
		\]
		
		\noindent
		As $\lambda > 1$ we have by~(H\ref{monot}) that $\kappa(r) > 0$ for $0 <r < R_{\lambda}$. Additionally, $k(z/\lambda) \leq k(z)$ for any $z \in \R^d$. Consequently, for any $0 <r < R_{\lambda}$,
		\begin{align*}
			\Ups(f) &\le\int k(y-x)\,d \gamma(x,y)\\
			&= \frac{m}{M}\int k\lt(\dfrac{y-x}\lambda\rt)\,d\gamma^{\lambda}(x,y)\\
			&= \frac{m}{M} \int \lt[ k\lt(\dfrac{y-x}\lambda\rt)-k(y-x) \rt] d\gamma^{\lambda}(x,y) + \frac{m}{M} \int k(y-x) \, d \gamma^{\lambda}(x,y) \\
			&\le \frac{m}{M} \int_{\Gamma_r} \lt[k\lt(\dfrac{y-x}\lambda\rt)-k(y-x) \rt] d\gamma^{\lambda}(x,y) + \frac{m}{M}\E(M).
		\end{align*}
		
		\noindent
		In the integral over $\Gamma_r$, the term in brackets is smaller than $-\kappa(r)$. Hence, for every $0 < r < R_{\lambda}$
		\begin{equation}\label{prf_prop_mu_inc_1}
			\E(m)-\eps\le \Ups(f) \leq \frac{m}{M}\E(M)- \frac{m}{M}\kappa(r)\gamma^{\lambda}(\Gamma_r).
		\end{equation}
		
		\noindent
		At this point we can send $\eps$ to 0 and deduce that $\E(m) \leq (m/M)\E(M)$. However we need to establish a strict inequality. For this we prove in the next step that there exist $r^*,\delta>0$ not depending on $\eps$ or $f$ such that $\gamma^\lambda(\Gamma_{r^*})\ge \delta$.
		
		\medskip
		
		\noindent
		\textit{Step 2. Conclusion.} 
		
		\noindent
		For $r\ge 0$, we set
		\[
		\overline{k}(r):=\max\{k(z):|z|\le r\}.
		\]
		
		\noindent
		This function is increasing,  continuous and there holds $\overline{k}(0)=0$.  Notice that using a ball of mass $m$ as a candidate for the energy $\E(m)$, we see that $\E(m) > 0$ for any $m>0$. Let us fix $0 < r^* < R$ such that
		\begin{equation}\label{prf_prop_mu_inc_2}
			m\overline{k}(r^*)\le \E(m)/4.
		\end{equation}
		
		\noindent
		By~\eqref{defR} we have  $|x-y| \leq R$ for $(x,y) \in \supp \gamma$ and by definition $|x-y|\leq r^*$ for $(x,y)\not\in\Gamma_{r^*}$. We deduce
		\begin{align*}
			\frac{\E(m)}{2} <\Ups(f) 
			\le \int_{\Gamma_{r^*}} c \, d\gamma + \int_{\Gamma_{r^*}^c} c \, d\gamma
			&\le\gamma(\Gamma_{r^*})\overline{k}(R)+\lt(m-\gamma(\Gamma_{r^*})\rt)\overline{k}({r^*})\\
			&\st{\eqref{prf_prop_mu_inc_2}}{\le}\gamma(\Gamma_{r^*})\overline{k}(R)+\lt(m-\gamma(\Gamma_{r^*})\rt)\frac{\E(m)}{4m}.
		\end{align*}
		\noindent
		This implies 
		\begin{equation*}
			\gamma(\Gamma_{r^*}) \lt(\overline{k}(R) - \frac{\E(m)}{4m}\rt) > \frac{\E(m)}{4}.
		\end{equation*}
		Thus  $4m\overline{k}(R) > \E(m)$ and
		\[
		\frac{m}{M}\gamma^{\lambda}(\Gamma_{\lambda r^*})=\gamma(\Gamma_{r^*})\ge\frac{m\E(m)}{4m\overline{k}(R)-\E(m)}=:m^*>0.
		\]
		
		\noindent
		Plugging this in~\eqref{prf_prop_mu_inc_1} with $r= \lambda r^*<R_\lambda$ we obtain
		\[
		\E(m)-\eps\le\frac{m}{M}\E(M)-m^*\kappa(\lambda r^*).
		\]
		
		\noindent
		Since $\eps\in[0,\E(m)/2]$ is arbitrary and $m^*\kappa(\lambda r^*)>0$, this proves the proposition.
	\end{proof}
	
	We close this subsection with a lemma establishing that if a function $f$ nearly maximises $\E(m)$ for some $m>0$ then there exists a cube which is at least half filled by $f$.
	
	\begin{lemma}\label{lem_concentration}
		Let $m>0$. There exists a non-decreasing function $r_0 : m \mapsto r_0(m)$ such that for $m>0$ and $f \in L^1_m$ with $\Ups(f)\ge \E(m)/2$, there exists a cube $Q_0$ of side-length $r_0(m)$ such that:
		\[
		\int_{Q_0}f\,\ge \frac{|Q_0|}2.
		\]
	\end{lemma}

	\begin{proof}~\\
		Let $r_0>0$ to be fixed later and assume by contradiction that there exists a partition $\mathcal{Q}$ of $\R^d$ in cubes with side-length $r_0$ such that for every $Q\in\mathcal{Q}$, 
		\begin{equation*}
			\int_Qf\,<\frac{|Q|}{2}.
		\end{equation*}
		The strategy to get a contradiction from this hypothesis is to build an exterior transport plan for $f$ with too small transport cost. Let $Q\in\mathcal{Q}$. Since $\smallint_Q (1-f)\,\ge\smallint_Qf\,$ there exists a function $g_Q\ge0$ supported in $Q$ such that $\smallint g_Q\,=\smallint_Qf\,$ and $f \chi_Q + g_Q \leq 1$. We then set
		
		\begin{equation*}
			\gamma_Q:= f \chi_Q \otimes \frac{g_Q}{\int_Q f\,} \qquad \text{and} \qquad \gamma := \sum_{Q \in \mathcal{Q}} \gamma_Q.
		\end{equation*}
		Notice that $\gamma$ is a valid competitor for $\Ups(f)$. Next for $R>0$, 	we define \[\overline{k}(R) := \max \{ k(x), \, |x| \le R\}.\]
		We compute:
		\begin{multline}\label{E(m)<kbar}
			0 < \frac{\E(m)}{2} \leq \Ups(f) \leq \sum_{Q \in \mathcal{Q}} \int_{Q \times Q} k(y-x) \, d \gamma_Q\\
			\le \overline{k}\lt(\sqrt{d}r_0\rt) \sum_{Q \in \mathcal{Q}} \int_Q f\, = \overline{k}\lt(\sqrt{d}r_0\rt)\int f\, \le \overline{k}\lt(\sqrt{d}r_0\rt) m.
		\end{multline}
		Remarking that $\overline{k}: \R_+ \to \R_+ $ is continuous at $0$, increasing and with $\overline{k}(0) = 0$, we set
		\[
		r_0 := \max \lt\{ r>0 :  \overline{k}\lt(\sqrt{d}r\rt) \leq \frac{\E(m)}{4m}\rt\}\ >0
		\]
		and obtain a contradiction with~\eqref{E(m)<kbar}. This concludes the proof.
	\end{proof}

	\subsection{Existence of a maximiser for~\eqref{MP}}\hfill
	
	In the following subsection, we assume that~(H\ref{cont}),(H\ref{cone})\&(H\ref{monot}) hold and prove the existence of maximisers for~\eqref{MP}.
	
	\medskip
	We only have to prove that maximising sequences for $\E(m)$ are tight. However our result is more precise. We obtain that if $f$ nearly maximises $\E(m)$ then almost all its mass concentrates in a closed ball with radius $R_*=R_*(m)$. In the limit, maximisers are supported in such balls.
	
	\begin{proposition}\label{prop_strong_tightness}
		Let $m > 0$. There exist $R_* = R_*(m)>0$, $\eps_0 = \eps_0(m)>0$ non-decreasing in $m$ with the following property. Let $0 < \eps \le \eps_0$ and let $f \in L^1_m$ such that $\smallint f = m$ and $\Ups(f) \geq \E(m) - \eps$, then up to a translation there holds
		\[
		\int_{\R^d\setminus B_{R_*}}f\,\le\frac{2m}{\E(m)}\eps.
		\]
	\end{proposition}
	
	\begin{proof}~\\
		\textit{Outline of the proof.}\\
		(Step 1) We start by using Lemma~\ref{lem_concentration} to get a collection $\mathcal{Q}_0$ of cubes $Q$ of side-length $r_0=r_0(m)$ such that  $\int_Q f+g\,\ge|Q|/2$. We denote $\Om_0:=\cup\mathcal{Q}_0$. We also consider the set $\Om$ obtained by thickening $\Om_0$ by adding the cubes closer than some distance $R=R(m)$. The real $R$ is chosen so that no mass of $f\chi_{\Om_0}$ is sent outside $\Om$ by any optimal exterior transport plan of $f$.\smallskip\\		
		(Step 2) We build an exterior transport plan for $f$ whose cost is very close to $\Ups(f \chi_\Om)$. \smallskip\\
		(Step 3) Next, we show that $\Om$ concentrates almost all the mass of $f$.  Using the strict superadditivity of $m \mapsto \E(m)$  and the previous step, we deduce that $m_\Om :=\smallint f \chi_\Om$ is close to $m$.\smallskip\\
		(Step 4) Eventually, we show that the distance between cubes in $\mathcal{Q}_0$ is uniformly bounded. As the cardinal of $\mathcal{Q}_0$ is also bounded, we conclude that the diameter of $\Om$ is bounded by a distance only depending on $m$.
		
		\medskip
		
		\noindent
		\textit{Step 1. Construction of a collection of cubes on which $\smallint_Q(f+g)\,\ge|Q|/2$.}
		
		Let $m>0$ and $f$ as in the statement of the proposition and assume that $\Ups(f)\ge \E(m)/2$ so that 
		\begin{equation}\label{defeps} 
			\eps:=\E(m)-\Ups(f)\le\E(m)/2. 
		\end{equation}
		Let  $\gamma$ be a minimiser for $\Ups(f)$ and let us set $g := \gamma_y$. Let $r_0$ and $Q_0$ be given by Lemma~\ref{lem_concentration}. We denote by $\hat{\mathcal{Q}}$ the regular partition of $\R^d$ into cubes of side-length $r_0$ such that $Q_0 \in \hat{\mathcal{Q}}$. For $j \geq 0$ to be fixed later, we set $r_j:=2^{-j}r_0$.  Considering the partition $\mathcal{Q}$ of $\R^d$ into cubes of side-length $r_j$ obtained by refining $\hat{\mathcal{Q}}$, we define $\mathcal{Q}_0$ as the subset formed by the elements $Q\in\mathcal{Q}$ such that
		\[
		\int_Q (f+g) \,\ge\frac{|Q|}4.
		\]
		
		\noindent
		We remark that $\mathcal{Q}_0$ is not empty since
		\[
		\int_Q(f+g)\,\ge\int_Qf\,\ge\frac{|Q|}2 
		\]
		for at least one of the $2^j$ sub-cubes of $Q_0$ in the partition $\mathcal{Q}$. 
		
		\medskip
		
		Let us define $\Om_0:=\cup\mathcal{Q}_0$. By Proposition~\ref{prop_Ups_max}, there exists $R = R(m)$ such that $|x-y| \leq R$ on $\supp \gamma$. We denote by $\mathcal{Q}_R$ the collection of cubes $Q \in\mathcal{Q}$ such that $d(Q, \Om_0)\le R$, and by $\Om$ their union. By construction, there holds $\gamma(\Om_0 \times\Om^c)=0$. We now define
		\[
		f_\Om:= f \chi_\Om, \qquad \text{ and }\qquad m_\Om :=\int f_\Om\,, 
		\]
		and we let $\gamma_\Om$ be an optimal exterior transport plan for $f_\Om$, that is $\gamma\in\varPi_{f_\Om}$ with $\int c\,d\gamma_\Om= \Ups(f_\Om)$. We then set $g_\Omega:=(\gamma_{\Om})_y$.
		\smallskip
		
		By Proposition~\ref{prop_Ups_max} again, we have (since $m_\Om\le m$) that 
		\begin{equation}\label{gammaOmega}
			\gamma_\Om(\Om_0 \times \Om^c) = 0.
		\end{equation}
		
		\noindent
		\textit{Step 2. Building a transport plan for $f$ whose cost is close to $\Ups(f_\Om)$.}
		
		In this step we modify $\gamma_\Om$ to build an exterior transport plan $\gamma$ for $f$ with a cost close to $\Upsilon(f_\Om)$.
		More precisely, we require that for some constant $C = C(r_j)>0$ with $C(r_j)\to 0$ as $r_j\to 0$,
		\begin{equation*}
			\int c \, d \gamma - \int c \, d\gamma_\Om \leq C(m-m_\Om).
		\end{equation*} 
		The proof is a refinement of the proof of the Lipschitz continuity of $\Ups$, see Proposition~\ref{prop_basic_ups}~(iii). In the following we define successively the plans $\gamma^0$, $\gamma^1$, $\gamma^2$, $\gamma^3$ which satisfy in particular
		\[
		\supp\gamma^0\subset\overline\Om\times\overline\Om,\qquad
		\supp\gamma^1\subset\overline\Om\times\overline{\Om^c},\qquad
		\supp\gamma^2\subset\overline{\Om\setminus\Om_0}\times\overline{\Om\setminus\Om_0},\qquad
		\supp\gamma^3\subset\overline{\Om^c}\times\overline{\Om^c}.
		\]
		First we set $\gamma^0:=\gamma_\Om\restr\Om\times\Om$ and denote $f^0:=\gamma^0_x$, $g^0:=\gamma^0_y$. We build the three remaining plans in the following substeps. These constructions will satisfy  
		\[
		(\gamma^1+\gamma^2)_x=f_\Om-\gamma^0_x=f\chi_\Om-f^0\qquad\text{and}\qquad\gamma^3_x=f-f_\Om=f\chi_{\Om^c}.
		\]
		We will set eventually $\widehat\gamma:=\gamma^0+\gamma^1+\gamma^2+\gamma^3$ which will be an admissible transport plan for $f$. The difficulty is to preserve the constraint $f+\widehat\gamma_y\le1$ while controlling the cost.
		
		\smallskip
		\noindent
		\textit{Step 2.a. Construction of $\gamma^1$.}
		
		Let us denote $\gamma_\Om^1:=\gamma_\Om\restr\Om\times\Om^c$,  $f_\Om^1:=(\gamma_\Om^1)_x$ and $g_\Om^1:=(\gamma_\Om^1)_y=\chi_{\Om^c}g_\Om$. We can not rule out the possibility that $f+g_\Om>1$ in some part of $\Om^c$ so that we cannot set $\gamma^1=\gamma_\Om^1$. However, we will transport as much as possible mass through $\gamma_\Om^1$. Let us define
		\[
		u:=(f+g_\Om-1)_+,
		\]
		which corresponds to the excess mass transported through $\gamma^1_\Om$. Using the convention $0/0=0$, we define $\gamma^1$ by 
		\[
		d\gamma^1(x,y):=\frac{g_\Om(y)-u(y)}{g_\Om(y)}d\gamma^1_\Om(x,y).
		\] 
		At this point, we have
		\begin{equation}\label{estim_cost_0+1}
			\int c\,d(\gamma^0+\gamma^1)\le \int c\, d\gamma_\Om=\Upsilon(f_\Om).
		\end{equation}
		Moreover setting $f^1:=\gamma^1_x$ and $g^1:=\gamma^1_y$, there holds $\supp g^1\subset\overline{\Om^c}$. Notice that since $f_\Om\le f$, by Corollary~\ref{coro_monot_f} we have $f_\Om+g_\Om\le f+g$, so that $g_\Om \le f+ g$ in $\Om^c$ which implies $g^1\le f+g$. Thus
		\begin{equation}\label{intQf+Q1}
			\int_Q (f+g^1)\,\le\int_Q (2f+g)\,<\frac{|Q|}2,\qquad\text{for every }Q\in \mathcal{Q}\setminus\mathcal{Q}_R,
		\end{equation}
		where we used  the definition of $\mathcal{Q}_0$ and the fact that $[\mathcal{Q}\setminus\mathcal{Q}_R] \cap \mathcal{Q}_0=\emptyset$.\\
		Let us compute for later use the mass from $\Om$ that  still requires to be transported. By construction
		\begin{equation}\label{massleft}
			\int_\Om (f -f^0-f^1) \,=\int\, d(\gamma^1_\Om-\gamma^1)=\int\frac{u(y)}{g_\Om(y)}\,d\gamma_\Om^1(x,y)
			=\int_{\Om^c}(f+g_\Om-1)_+\,\le\int_{\Om^c}f\,.
		\end{equation}
		
		\noindent
		\textit{Step 2.b. Construction of $\gamma^2$.}
		
		We now define 
		\[
		\gamma^2_\Om := \gamma^1_\Om - \gamma^1 = \gamma_\Om-\gamma^0 - \gamma^1.
		\]
		Notice that by~\eqref{gammaOmega}, $\gamma^1_\Om(\Om_0 \times \Om^c) = 0$, so that
		\[
		\supp \gamma^2_\Om\subset\overline{\Om\setminus\Om_0}\times\overline{\Om^c}. 
		\]
		In particular, $f^2:=f_\Om-f^0-f^1$ is supported in $\overline{\Om\setminus\Om_0}$. Let $Q\in\mathcal{Q}_R\setminus\mathcal{Q}_0$. Since $g^0\le g_\Om$ and $f=f_\Om$ on $Q$, using  Corollary~\ref{coro_monot_f} again we see that
		\[
		\int_Q(f+g^0)\,\le\int_Q(f_\Om+g_\Om)\,\le\int_Q(f+g)\,\le\frac{|Q|}4.
		\]
		Therefore for such $Q$ there exists a function $g^2_Q:Q\to\R_+$ such that $f+g^0+g^2_Q\le1$ and $\int g^2_Q\,=\int_Qf^2\,$. Defining
		\[
		\gamma^2_Q:=\frac1{\int_Qf^2\,}\lt[\chi_Q f^2\rt]\otimes g^2_Q\quad\text{ for }Q\in\mathcal{Q}_R\setminus\mathcal{Q}_0\quad\text{ and then } \quad\gamma^2:=\sum_{Q\in\mathcal{Q}_R\setminus\mathcal{Q}_0} \gamma^2_Q,
		\]
		we have $\gamma^2_x=f^2$ and $g^2:=\gamma^2_y=\sum_Q g^2_Q$. Hence
		\begin{equation}\label{sumO2le1}
		f+g^0+g^2\le 1
		\end{equation}
and
		\[
		\int c\,d\gamma^2 \le\lt(\int f^2\,\rt)\overline{k}\lt(\sqrt dr_j\rt),
		\]
		where as in the proof of Proposition~\ref{prop_mu_inc} we denote $\overline{k}(r):=\max\{k(x):|x|\le r\}$.\\
		By construction $f^2=f_\Om-f^0-f^1$, so by~\eqref{massleft} there holds $\int f^2\,\le m-m_\Om$ which leads to the cost estimate
		\begin{equation}\label{estim_cost_2}
			\int c\,d\gamma^2\le(m-m_\Om)\overline{k}\lt(\sqrt d r_j\rt).
		\end{equation}
		\smallskip
		
		\noindent
		\textit{Step 2.c. Construction of $\gamma^3$.}
		
		We still have to transport the mass corresponding to $\chi_{\Om^c}f$. For every $Q\in\mathcal{Q}\setminus\mathcal{Q}_R$ we have $\int_Qf\,\le\int_Q(f+g)\,\le|Q|/4$, therefore, in view of~\eqref{intQf+Q1}, there exists a function $g^3_Q:Q\to\R_+$ such that $\int g^3_Q\,=\int_Q f\,$ and $f+g^1+g^3_Q\le1$. As in the previous step, we define
		\[
		\gamma^3_Q:=\frac1{\int_Qf\,}\lt[\chi_Q f\rt]\otimes g^3_Q \qquad\text{and}\qquad
		\gamma^3:=\sum_{Q\in\mathcal{Q}\setminus\mathcal{Q}_R} \gamma^3_Q.
		\]
		By construction, $(\gamma^3)_x=\chi_{\Om^c}f$ and denoting $g^3:=(\gamma^3)_y$, we have $\supp g^3\subset\overline{\Om^c}$ as well as
		\begin{equation}\label{sum13le1}
		f+g^1+g^3\le1.
		\end{equation}
Moreover,
		\begin{equation}\label{estim_cost_3}
			\int c\,d\gamma^3\le\lt(\int_{\Om^c}f\,\rt)\overline{k}\lt(\sqrt dr_j\rt)=(m-m_\Om)\overline{k}\lt(\sqrt dr_j\rt).
		\end{equation}
		\smallskip
		
		\noindent
		\textit{Step 2.d. Conclusion : definition and properties of $\widehat\gamma$.}
		
		Eventually, we set $\widehat\gamma:=\gamma^0+\gamma^1+\gamma^2+\gamma^3$ and $\widehat g:=\widehat\gamma_y$.  Recalling that $ \supp (g^0+g^2)\subset \overline{\Om}$, $\supp(g^1+g^3)\subset \overline{\Om}^c$, \eqref{sumO2le1} and \eqref{sum13le1}, there holds $\widehat\gamma_x=f^0+f^1+f^2+f^3=f$ and $f+\widehat g\le1$ so that $\widehat\gamma$ is an admissible exterior transport plan for $f$. Besides, collecting the estimates~\eqref{estim_cost_0+1},\eqref{estim_cost_2}\&\eqref{estim_cost_3} we get
		\begin{equation}\label{Upsf<UpsfOmega+}
			\Upsilon(f)\le\int c\, d\widehat\gamma\le\Upsilon(f_\Om)+2(m-m_\Om)\overline{k}\lt(\sqrt dr_j\rt).
		\end{equation}
		\smallskip
		
		\noindent
		\textit{Step 3. We show that $m - m_\Om \leq C(m) \eps$ (recall the definition~\eqref{defeps} of $\eps$).}
		
		As $\Ups(f_\Om) \leq \E(m_\Om)$ and $\E(m) - \eps =\Ups(f)$,~\eqref{Upsf<UpsfOmega+} yields
		\[
		\E(m)-\eps\le \E(m_\Om)+ 2 \overline{k}\lt(\sqrt{d} r_j\rt)(m-m_\Om).
		\]
		
		\noindent
		Additionally by Proposition~\ref{prop_mu_inc}, $\E(m_\Om)  \leq \frac{m_\Om}{m} \E(m)$. Hence
		\[
		\lt(\frac{\E(m)}{m}- 2 \overline{k}\lt(\sqrt{d} r_j\rt)\rt)(m-m_\Om)\le \eps.
		\]
		
		\noindent
		By continuity of $k$, $\overline{k}(\sqrt{d} r_j) \to 0$ as $r_j \to 0$. Recalling that $r_j = 2^{-j} r_0$, we fix $j\ge0$ as the first integer such that $\overline{k}(\sqrt{d} r_j) \le \E(m)/4m$ (notice that $j$ does not depend on $\eps$). Therefore
		\begin{equation}\label{prf_lem_tight_3}
			m-m_\Om\le \dfrac{2m\eps}{\E(m)}.
		\end{equation}
		This yields
		\begin{equation}\label{prf_lem_tight_35}
			\int_{\R^d \setminus\Om}f \,=\int f\, - \int f_\Om\, = m-m_\Om\le\dfrac{2m \eps}{\E(m)}.
		\end{equation}
		
		\noindent
		For future use, let us also notice that injecting~\eqref{prf_lem_tight_3} into~\eqref{Upsf<UpsfOmega+} we obtain
		\begin{equation}\label{energyf}
			\E(m) - \eps = \Ups(f) \leq \Ups(f_\Om) + \eps.
		\end{equation}
		
		\medskip
		
		\noindent
		\textit{Step 4 : Bounding the diameter of $\Om$.}
		
		We finally prove that $\Om$ is uniformly bounded which would conclude the proof.
		For $Q_-$, $Q_+\in\mathcal{Q}_0$, we write $Q_-\sim Q_+$ if there exists a finite chain
		\begin{equation}\label{prf_lem_tight_4}
			Q_-=Q_0,Q_1,\dots,Q_n=Q_+
		\end{equation}
		
		\noindent
		such that $Q_i\in\mathcal{Q}_0$ and $d(Q_{i-1},Q_i)\le 4R +\sqrt{d}r_j$ for $1\le i\le n$. This defines an equivalence relation. Let us show that there exists only one equivalence class. We assume by contradiction that there exist at least two equivalence classes, and we let $\mathcal{C}^1$ be one of these classes and $\mathcal{C}^2$ be the union of the remaining classes. For $i \in \{1, 2\}$, we then define $\Om^i$ to be the union of the cubes $Q$ such that $d(Q, \mathcal{C}^i)\le R$. By construction, $d(\Om_1,\Om_2)>2R$. Recalling that $\Om$ is the union of the cubes $Q$ such that $d(Q, \Om_0) \leq R$, we have $\Om^1\cup\Om^2=\Om$.
		
		For $i \in \{1, 2\}$, we set $f_\Om^i:= f_\Om \chi_{\Om^i}$ and $m_\Om^i=\int f_\Om^i\,$. We have $m_\Om^1+m_\Om^2=m_\Om\le m$ and $m_\Om^1,m_\Om^2\ge 2^{-jd}|Q_0|/4=2^{-jd-2}|Q_0|$. Additionally, by Proposition~\ref{prop_basic_ups} $(ii)$,
		\[
		\Ups(f_\Om)=\Ups(f^1_\Om)+\Ups(f^2_\Om)\le \E(m^1_\Om)+\E(m^2_\Om).
		\]
		
		\noindent
		Injecting this inequality into~\eqref{energyf} yields
		\[
		\E(m) - \eps = \Ups(f) \leq \Ups(f_\Om) +\eps\leq \E(m^1_\Om)+\E(m^2_\Om)+\eps.
		\]
		
		\noindent
		Recalling that  $e(m)=\E(m)/m$, this rewrites as
		\begin{equation}\label{ineqmu}
			m e(m) \leq m^1_\Om e(m^1_\Om)+m^2_\Om e(m^2_\Om)+2\eps.
		\end{equation}
		
		\noindent
		As $m^1_\Om+m^2_\Om\le m$ and for $i \in \{1, 2\}$, $m^i_\Om\ge2^{-jd-2}|Q_0|$, we have $m^i_\Om \leq m - 2^{-jd-2}|Q_0|$. Recall that by Proposition~\ref{prop_mu_inc}, $e$ is increasing, so that $e(m^i_\Om) \leq e(m-2^{-jd}m_0)$. Hence
		\[
		m^1_\Om e(m^1_\Om)+m^2_\Om e(m^2_\Om) \leq m e\lt(m-2^{-jd}m_0\rt).
		\]
		
		\noindent
		With~\eqref{ineqmu}, we obtain
		\[
		m e (m) \leq me\lt(m-2^{-jd}m_0\rt) + 2\eps,
		\] 
		
		\noindent
		which is absurd for $\eps$ small enough because $e$ is increasing. It follows that for $\eps>0$ small enough the relation $\sim$ has a single class. Recall that for all $Q \in \mathcal{Q}_0$, $\smallint_Q (f+g)\, \,\geq 2^{-jd-2}|Q_0|$. Thus the maximal length of a chain in~\eqref{prf_lem_tight_4} without any repetition is bounded by $N := \lfloor 2^{jd+3}m/|Q_0|\rfloor$. Therefore, the diameter of $\Om$ is bounded by $(4R +2\sqrt{d}r_j)(N+1)$ with $r_j$ and $N$ only depending on $m$, the dimension $d$ and the cost $c$. Together with~\eqref{prf_lem_tight_35} this proves the proposition.
	\end{proof}
	
	We can now apply the direct method of Calculus of Variations to establish the existence of a maximiser for~\eqref{MP}.
	
	\medskip
	
	\begin{proof}[Proof of Theorem~\ref{thm_MP_max}]~\
		
		Let  $f_n$ be a maximising sequence for~\eqref{MP} and let $R_* = R_*(m)$ be given by Proposition~\ref{prop_strong_tightness} so that  up to translation,
		\begin{equation*}
			\int_{\R^d \setminus B_{R_*}} f_n\,  \to 0 \quad \text{as} \quad n \to \infty.
		\end{equation*}
		
		\noindent
		Therefore, $f_n$ is a tight sequence of $\mathcal{M}_+(\R^d)$ and up to extraction of a subsequence it converges weakly-$*$ to $f$ where  $f$ is admissible for~\eqref{MP}. By Proposition~\ref{prop_basic_ups} $(iv)$,
		
		\begin{equation*}
			\Ups(f) = \lim \Ups(f_n) = \E(m),
		\end{equation*}
		
		\noindent
		so that $f$ is a maximiser for $\E(m)$.
		
		\medskip
		
		Let now $f$ be any maximiser of $\E(m)$. Applying Proposition~\ref{prop_strong_tightness} to $f$ we have that up to a translation $\supp f \subset \overline B_{R_*}$. This concludes the proof.
	\end{proof} 
	Let us show that when $f$ is compactly supported there exist Kantorovitch potentials for the problem~\eqref{PPRd} (this is the situation of interest as we have just established that the maximisers of $\E(m)$ are compactly supported in $\R^d$).
	\begin{lemma}\label{lem_ctransfglobal}
		Let $m >0$ and assume that $f \in L^1_m$ is compactly supported. Let $R = R(m)$ be given by Proposition~\ref{prop_Ups_max} such that all minimisers $\gamma$ of $\Ups(f)$ are supported in $X:= \supp f + \overline{B}_R$. Then, there exists a pair $(\vhi, \psi) \in C_c(\R^d) \times C_c(\R^d)$ optimal for $\Ups^*(f)$. Additionally, $\vhi = \psi^c$, $\psi = \vhi^c_{\,-}$ and both $\vhi$ and $\psi$ are compactly supported in $X$.
	\end{lemma}
	
	\begin{proof}
		Let us introduce $\tilde c := c_{|X \times X}$ which is a continuous cost function on the compact set $X$. By Proposition~\ref{prop_ex_DPp}, there exists $\widetilde \psi \in C(X)$ with $\widetilde \psi = ({\widetilde \psi}^{\tilde c \tilde c})_-$ such that $\Ups^*(f) = K_f({\widetilde \psi}^{\tilde c}, \widetilde \psi)$. By Proposition~\ref{prop_Ups_OT},
		\[
		\{{\widetilde \psi}^{\tilde c \tilde c} <  0 \} \subset \{f+g=1\}\subset X.
		\]
		
		\noindent
		Combining this with  $\widetilde \psi = ({\widetilde \psi}^{\tilde c \tilde c})_-$ and $f+g=0$ on $\partial X$ , we get $\widetilde \psi=0$ on $\partial X$.  We extend the potentials on $\R^d$ by setting
		
		\[
		\psi := \begin{cases}
			\widetilde \psi  &\text{in} \quad X,\\
			0  &\text{in} \quad X^c,
		\end{cases}
		\qquad \text{and for }x\in\R^d,\quad \vhi(x) := \psi^c(x) = \inf\{c(x,y) - \psi(y) : \, y \in \R^d \}.
		\]
		
		\noindent
		We now show that the pair $(\vhi, \psi)$ satisfies the conclusion of the lemma.
		
		Observe that $\psi$ is continuous and supported in $X$ and that $\psi \leq 0$. Hence $\vhi \geq 0$. Moreover, for $x \in\R^d\setminus\supp \psi$,
		\[
		\vhi(x) \leq c(x,x) - \psi(x) = 0,
		\]
		
		\noindent
		so that $\vhi$ is also  supported in $X$.\\
		Next, for $x \in X$,
		\[
		\vhi(x) = \min \lt( \min_{y \in \R^d\setminus X} c(x,y), \ \min_{y \in X} \{c(x,y) - \widetilde \psi(y) \} \rt).
		\]
		
		\noindent
		Let $x \in X$. For $y \in \R^d\setminus X$, there exists $\widetilde y$ in the intersection of the segment $[x,y]$ with $\partial X$. By continuity, $\widetilde \psi(\widetilde y) = 0= \widetilde\psi(y)$ and moreover by~(H\ref{monot}),  $c(x, \widetilde y) \leq c(x, y)$ so that $c(x, \widetilde y)- \widetilde \psi(\tilde y) \le c(x,y)$. We deduce that for $x \in X$ the above formula simplifies as
		\[
		\vhi(x)  = \min_{y \in X} \{c(x,y) - \widetilde \psi(y) \} = {\widetilde \psi}^{\tilde c}(x).
		\]
		
		\noindent
		This proves that $\vhi$ is continuous and that $K_f(\vhi, \psi) = K_f({\widetilde \psi}^{\tilde c}, \widetilde \psi) = \Ups^*(f)$. Moreover, using the same argument as above, we have $\vhi^c(y) = 0$ for $y \not \in X$. For $y \in X$,
		\[
		\vhi^c(y) = \min \lt( \min_{x \in \R^d\setminus X} c(x,y), \ \min_{x \in X} \{c(x,y) - {\widetilde \psi}^{\tilde c}(x) \} \rt) = \min_{x \in X} \{c(x,y) - {\widetilde \psi}^{\tilde c}(x) \} = {\widetilde \psi}^{\tilde c \tilde c}(y).
		\]
		
		\noindent
		We deduce $\vhi^c_{\,-} = 0 = \psi$ in $X^c$, and $\vhi^c_{\,-} = (\psi^{\tilde c \tilde c})_- = \widetilde \psi$ in $X$. Thus $\vhi^c_{\,-} = \psi$ everywhere. This ends the proof of the lemma.
	\end{proof}
	Let us now recall a variant of the bathtub principle, see~\cite[Theorem 1.14]{LiebLoss}.
	
	\begin{proposition}\label{prop_bathtub}
		Let $\xi : \R^d \to \R_+$ be measurable and such that  for all $t \ge 0$, $|\{ \xi > t\}| < \infty$. Given  $m >0$, let
		\[
		t := \inf \{ s \ge 0, \, |\{ \xi > s \}| \leq m\}.
		\]
		Then, the maximisers of
		\[
		\sup_{\widetilde{f}} \lt\{\int \widetilde{f} \xi \, : \widetilde{f} \in L^1(\R^d), \, 0 \leq \widetilde{f} \leq 1, \, \int \widetilde{f}\, =m\rt\}
		\]
		are the functions $f := \chi_{\{\xi > t \}} + \theta$, where $\theta \in L^1(\R^d, [0,1])$ is supported in $\{\xi = t\}$ and satisfies
		\[
		\int \theta\,= m-|\{\xi>t\}|.
		\]
	\end{proposition}
	
	We are now ready to establish Corollary~\ref{coro_maxE}.
	
	\smallskip
	
	\begin{proof}[Proof of Corollary~\ref{coro_maxE}]~\

		By Theorem~\ref{thm_MP_max}, the optimisation problem~\eqref{MP} admits a compactly supported solution $f$. Let $(\vhi, \psi) \in C_c(\R^d)\times C_c(\R^d)$ be an optimal pair for $\Ups^*(f)$ provided by Lemma~\ref{lem_ctransfglobal}, so that
		\[
		\Ups^*(f)= \int f(\vhi - \psi)\, + \int \psi\,. 
		\]
		We see that $f$ is a maximiser of:
		\[
		\sup\lt\{ \int \widetilde{f}(\vhi - \psi)\,: \widetilde{f} \in L^1(\R^d), \, 0 \leq \widetilde{f} \leq 1, \, \int \widetilde{f}\, = m\rt\}.
		\]
		Let us set $\xi := \vhi - \psi\ge 0$. By Proposition~\ref{prop_bathtub} there exists $t \ge 0$ and $\theta \in L^1(\R^d, [0,1])$ supported in $\{\xi =t \}$ such that $f  = \chi_{\{\xi> t \}} + \theta$.
		Notice in particular that since $\theta\in [0,1]$, we have
		\[
		|\{\xi=t\}|\ge \int \theta \, = m - |\{\xi> t \}|.
		\]
		and there exist measurable subsets $G\subset \{\xi = t\}$ with $|G|=m- |\{\xi> t \}|$. For any such set, setting
		\[
		\bar f := \chi_{\{\xi> t \}} + \chi_G,
		\]
		we have $\Ups^*(\bar f) = \Ups^*(f)$ and $\bar f$ is also a maximiser of~\eqref{MP}. Since $\bar f$ is a characteristic function, by Theorem~\ref{thm_saturation} and Corollary~\ref{coro_unique_g}, there exists $F \subset \R^d$ such that any minimiser $\gamma$ of $\Ups({\bar f})$ satisfies $\gamma_y = \chi_F$. Setting $E := \{\xi> t \} \cup G$, we deduce that
		\[
		\Ups_{\mathrm{set}}(E) =\Ups^*(\bar f) =\Ups(\bar f) = \Ups(f)
		\] 
		so that  $\E(m) = \E_{\mathrm{set}}(m)$, which concludes the proof.
	\end{proof}
	
	\section{Maximisers of~\eqref{MP} are characteristic functions of balls}\label{ballisunique}

	In this section we prove Theorem~\ref{thm_ball_unique_max}. We assume that $c(x,y)=k(|y-x|)$ with $k\in C(\R_+,\R_+)$ increasing and coercive and with $k(0)=0$. In particular, we have now $c = \bar c$, so that the operations of $c$-transform and $\bar c$-transform coincide. Also notice that by Theorem~\ref{thm_saturation}, if $f = \chi_E$ for some Lebesgue measurable set $E$ then $\Ups_{\mathrm{set}}(E) = \Ups(\chi_E)$. By abuse of notation, we write $\Ups(E)$ for $\Ups(\chi_E)$. Since the class of costs that we consider is invariant by scaling we assume without loss of generality that  $m = \om_d$.
	
	\medskip

	We now recall the definition of symmetric rearrangement of functions with constant sign (see~\cite[Chapter 3]{LiebLoss} for more details on symmetric rearrangements).
	
	\begin{definition}~\label{def_sym_rear}
		\begin{enumerate}[(i)]
			\item Given a measurable set $A \subset \R^d$, we define the symmetric rearrangement of $A$ as the open ball $A^*$ centred at the origin and of volume $|A|$. 
			
			\item Let $\vhi : \R^d \to \R_+$ be measurable and such that for every $t\geq 0$, $|\{\vhi > t\}| <\infty$. Its symmetric decreasing rearrangement is defined by
			\begin{equation*}
				\vhi^*(x) := \int_{\R_+} \chi_{\{ \vhi > t \}^*} (x) \, dt.
			\end{equation*}	
			
			\item Let $\psi : \R^d \to \R_-$ be measurable and such that for every $t\leq 0$, $|\{\psi < t\}| <  \infty$. Its symmetric increasing rearrangement is defined by
			\begin{equation*}
				\psi_*(x) :=-(-\psi)^*(x)= - \int_{\R_-} \chi_{\{ \psi < t \}^*} (x) \, dt.
			\end{equation*}	
		\end{enumerate}	
	\end{definition}
	
	The following lemma recalls some basic properties of the symmetric increasing rearrangement $\psi_*$ of a non-positive function $\psi$. All these properties  but the continuity of $\psi_*$ follow immediately from the definition. The fact that continuity is preserved by symmetric rearrangement is well-known but we have no reference for this at hand. We provide a short proof for the reader's convenience.
	
	\begin{lemma}\label{lem_psi_star}
		Let $\psi : \R^d \to \R_-$ be as in Definition~\ref{def_sym_rear}. Then, $\psi_*$ is non-positive, radial, non-decreasing, and for any $t \leq 0$, $\{\psi_*<t\}=\{\psi<t\}^*$. Besides, if $\psi$ is supported in a compact set of diameter bounded by $2R >0$ then $\psi_*$ is  supported in $\overline B_R$. If moreover $\psi$ is continuous then $\psi_*$ is also continuous.
	\end{lemma}
	
	\begin{proof}[Proof of the last point]
		Let $\psi\in C_c(\R^d,\R_-)$. First, as the strict sublevels sets $\{\psi_* < t\}$ are the open balls $\{\psi < t\}^*$, $\psi_*$ is upper semi-continuous (note that this is true even when $\psi$ is not continuous).
		
		Let us now establish that $\psi_*$ is lower semi-continuous, \ie that for any $t \leq 0$, $\{ \psi_* \leq t \}$ is closed. We first notice that $\{\psi_* \leq 0 \} = \R^d$ is closed. Given $t < 0$, let $t_n<0$ be a decreasing sequence converging to $t$. Observe that if for some $n \geq 0$,  $\{\psi < t_n\} = \emptyset $, then $\{\psi_* \le t\} = \emptyset$ is closed. Next, we assume that for every $n \geq 0$,
		\begin{equation}\label{sublvlpsi}
			\{\psi < t_n \} \neq \emptyset.
		\end{equation}
		
		\noindent
		We denote by $R_n$ the radius of the ball $\{\psi_* < t_n\}$. Notice that the sequence $R_n$ is non-increasing and bounded by $0$, so that $R_n$ converges to some $R \geq 0$. 
		
		Let us show that the sequence $R_n$ is decreasing. By contradiction, we assume that $R_n = R_{n+1}$ for some $n \geq 0$. Then $\{\psi < t_n\}^* = \{\psi < t_{n+1}\}^*$ and $|\{t_{n+1} \leq \psi < t_n\}| = 0$. Using~\eqref{sublvlpsi} and the fact that $\psi$ is compactly supported, there exists $x$ such that $\psi(x)<t_{n+1}$ and $y$ such that $\psi(y)>t_n$. 
		Thus by continuity of $\psi$ there exists $z$ such that $\psi(z) = (t_{n+1} + t_n)/2$. By continuity of $\psi$ again, there exists $\eta >0$ such that $B_\eta(z) \subset \{t_{n+1} < \psi < t_n\}$, contradicting the fact that $|\{t_{n+1} \leq \psi < t_n\}| = 0$. As a conclusion, the sequence $R_n$ is decreasing and
		\[
		\{\psi_* \leq t \} = \bigcap_{n \geq 0} \{\psi_* < t_n\} = \bigcap_{n \geq 0} B_{R_n} = \overline{B}_{R}.
		\]
		
		\noindent
		Hence $\psi_*$ is lower semi-continuous and therefore  continuous.
	\end{proof}
	
	To prove Theorem \ref{thm_ball_unique_max}, we need a last lemma  characterising optimal potentials of $\Ups^*(\chi_{B_1})$.  Along the way we will prove that the set $F$ minimizing $\Ups_{\mathrm{set}}(E)$ (recall \eqref{Upsfunc}) is the annulus $A=B_{2^{1/d}}\backslash B_1$.
	
	\begin{lemma}\label{lem_monotpotball}
		Let $(\psi^c, \psi)$ be a pair of optimal potentials for $\Ups^*(\chi_{B_1})$ such that $\psi$ is radially symmetric and non-decreasing. Then $\psi^c$ is radially symmetric and non-increasing. Besides, $\psi^c$ is radially decreasing on $B_{1}$. Finally, if $\gamma$ is a minimizer of $\Ups(\chi_{B_1})$ then $\gamma_y=\chi_{A}$.
	\end{lemma}
	
	\begin{proof}
		Combining the facts that  $k$ is continuous, that $k(r) \to \infty$ as $r \to \infty$ and that $\psi$ is bounded by Lemma \ref{lem_ctransfglobal}, we see that for any $x \in \R^d$,
		\[
		\psi^c(x) = \min \{ k(|y-x|) - \psi(y) : y \in \R^d \}.
		\]
		As $\psi$ is radially symmetric non-decreasing and $k$ is increasing, we easily see that
		\begin{equation}\label{altinfdef}
			\psi^c(x) = \min \{k(|y-x|) - \psi(y) : y, \, \exists \lambda \geq 1, \, y = \lambda x\},
		\end{equation}
		which in turn implies that $\psi^c$ is radially symmetric.
		
		\medskip
		
		From now on, for radial functions $\zeta:\R^d\to\R$, we make the abuse of notation $\zeta(r)=\zeta(r\sigma)$ for $r\ge0$ where $\sigma$ is some fixed element of $\mathbb{S}^{d-1}$. With this convention~\eqref{altinfdef} reads
		\begin{equation}\label{altinfdef_bis}
			\psi^c(r) = \min_{s\ge r} k(s-r) - \psi(s).
		\end{equation}
		Let us prove that $\psi^c$ is non-increasing. Let $0 \le r_1 \le r_2$. By~\eqref{altinfdef_bis}, there exists $r\ge r_1$ such that
		\begin{equation}\label{psicx1}
			\psi^c(r_1) = k(r-r_1) - \psi(r).
		\end{equation}
		If $r \leq r_2$, we use $\psi^c(r_2)\le k(0) -\psi(r_2)=-\psi(r_2)$ and deduce from~\eqref{psicx1} and the fact that $\psi$ is non-decreasing that
		\begin{equation*}
			\psi^c(r_2) - \psi^c(r_1) \leq  \psi(r) - \psi(r_2) -k(r-r_1) \leq 0.
		\end{equation*}
		If $r > r_2$, we use $\psi^c(r_2)\le k(r-r_2)-\psi(r)$ to get
		\begin{equation*}
			\psi^c(r_2) - \psi^c(r_1) \leq  k(r-r_2) - k(r-r_1)  \leq 0,
		\end{equation*}
		because $r_1\le r_2<r$ and $k$ is increasing. In both cases $\psi^c(r_2) - \psi^c(r_1) \leq 0$. Hence $\psi^c$ is non-increasing on $\R^d$.
		
		\medskip
		
		We now prove that $\psi^c$ is decreasing on $B_{1}$. Let $0 < r_1 < r_2 < 1$. Given $\gamma$ a minimiser for $\Ups(B_{1})$, there exists $y\in \R^d\setminus B_1$ such that $(r_1 \sigma, y) \in \supp \gamma$. By Proposition~\ref{prop_Ups_OT}, $\gamma$ is an optimal transport plan between $f:=\chi_{B_1}$ and $g := \gamma_y$ and $(\psi^c, \psi^{cc})$ is a pair of Kantorovitch potentials for the transport between $f$ and $g$. Therefore,
		\begin{equation}\label{y1min}
			\psi^c(r_1) + \psi^{cc}(|y|) = k(|y-r_1 \sigma|).
		\end{equation}
		Let us prove by contradiction that $y\in [1,+\infty)\sigma$. Assume it is not and let $y' := |y|\sigma$. Recalling that $|y|\ge1>r_1$, we have $|y'-r_1 \sigma| = |y|-r_1 < |y - r_1 \sigma|$ and since $k$ is increasing we deduce
		\[ 
		k(|y'-r_1 \sigma|) < k(|y - r_1 \sigma|).
		\]
		Then, by definition of $\psi^{cc}$ and taking into account that it is radially symmetric we get
		\[
		\psi^c(r_1) + \psi^{cc}(|y|) \leq k(|y'-r_1 \sigma|)< k(|y - r_1 \sigma|)
		\]
		which contradicts~\eqref{y1min}. Therefore, $y = r \sigma$ for some $r \geq 1$. By definition of the $c$-transform,
		\begin{equation}\label{defKP}
			\psi^c(r_2) + \psi^{cc}(r) \leq k(r-r_2).
		\end{equation}
		Subtracting~\eqref{y1min} to \eqref{defKP}, we obtain
		\begin{equation*}
			\psi^c(r_2) - \psi^c(r_1) \leq k(r-r_2) - k(r-r_1) < 0,
		\end{equation*}
		where we used $r_1 < r_2 <1\le r$. This shows that $\psi^c$ is decreasing on $B_{1}$.\\
		Finally we notice that as a consequence of the above discussion, the plan $\gamma$ is radial. Combining this with Lemma \ref{lem_saturation} proves that $g=\chi_{A}$.
	\end{proof}
	We are now ready to prove Theorem~\ref{thm_ball_unique_max}.
	\smallskip

	\begin{proof}[Proof of Theorem~\ref{thm_ball_unique_max}]~\
		\begin{center}
			\textit{Part I : Unit balls are maximisers of $\E(\om_d)$.}
		\end{center}
		
		By Theorem~\ref{thm_MP_max}, there exists a compactly supported maximiser $f$ for~\eqref{MP} with $m = \om_d$. By Lemma~\ref{lem_ctransfglobal}, there exists an optimal pair $(\psi^c, \psi) \in C_c(\R^d) \times C_c(\R^d)$ for problem $\Ups^*(f)$ such that $\psi = (\psi^{cc})_-$.
		
		\medskip
		
		\noindent
		\textit{Step 1. We build a radially symmetric maximiser for~\eqref{MP}.}
		
		Let $\psi_*$ be the symmetric increasing rearrangement of $\psi$. By Lemma~\ref{lem_psi_star}, as $\psi \in C_c(\R^d)$, we also have $\psi_* \in C_c(\R^d)$. We denote by $\psi_*^{\,c}$ the function $(\psi_*)^c$. By definition, $\psi_*^{\,c} \oplus \psi_* \leq c$. Proceeding as in the proof of Lemma~\ref{lem_ctransfglobal}, we obtain $\psi_*^{\,c} \in C_c(\R^d)$. Thus $(\psi_*^{\,c}, \psi_*)$ is admissible for $\Ups^*(B_{1})$.
		
		Notice that $(f, \psi)$ solves the double supremum problem (recall the definition~\eqref{Kf} of $K_f$)
		\[
		\sup_f \sup_{\psi \in C_c(\R^d)} \lt\{K_f(\psi^c, \psi) : 0\leq  f \leq 1, \, \int f \,= \om_d, \, \psi \leq 0 \rt\}.
		\]
		
		\noindent
		Hence
		\[
		\E(\om_d) = K_f(\psi^c, \psi) =  \int f (\psi^c - \psi) \,+\int \psi\, \geq \Ups^*(B_{1}) \geq K_{\chi_{B_{1}}}(\psi_*^{\,c}, \psi_*)  = \int_{B_{1}} (\psi_*^{\,c}-\psi_*)\, + \int \psi_*\,.
		\]
		\noindent
		In the remainder of this step, we establish the converse inequality 
		\begin{equation}\label{converseineq}
			K_f(\psi^c, \psi) \leq K_{\chi_{B_{1}}}(\psi_*^{\,c}, \psi_*),
		\end{equation}
		
		\noindent 
		so that  $B_1$ is a  maximiser of $\E(\om_d)$ and the first part of Theorem~\ref{thm_ball_unique_max} is proved.  Notice that~\eqref{converseineq} also implies that $(\psi_*^{\,c}, \psi_*)$ is a pair of optimal potentials for $\Ups^*(B_{1})$. To establish~\eqref{converseineq}, we first notice that by construction $\smallint \psi = \smallint \psi^*$ so that we only need to prove
		\begin{equation}\label{toprovereduced}
			\int f (\psi^c - \psi)\, \leq \int_{B_{1}} (\psi_*^{\,c}-\psi_*)\,.
		\end{equation}
		In Step 2 below we establish the inequality 
		\begin{equation}\label{psistarccstar}
			(\psi^c)^* \leq (\psi_*)^c = \psi_*^{\,c},
		\end{equation}
		
		\noindent
		where $(\psi^c)^*$ denotes the symmetric decreasing rearrangement of $\psi^c$. Admitting that~\eqref{psistarccstar} holds we deduce~\eqref{toprovereduced} as follows. Since $f$ is non-negative and compactly supported we have by the Hardy-Littlewood inequality (see~\cite[Theorem 3.4]{LiebLoss})
		\begin{equation}\label{ineqHL}
			-\int f \psi \,\leq - \int f^* \psi_* \,\qquad \text{and} \qquad \int f \psi^c\, \leq \int f^* (\psi^c)^*\, \st{\eqref{psistarccstar}}{\leq} \int f^* \psi_*^{\,c}\,.
		\end{equation}
		
		\noindent
		Using that  $-\psi_*$ and $\psi_*^{\,c}$ are radially symmetric and non-increasing, we may appeal to   Proposition~\ref{prop_bathtub} and conclude that separately,
		\begin{equation}\label{lefttoprove}
			- \int f \psi \,\leq - \int \chi_{B_1} \psi_*\,  \qquad \textrm{and } \qquad \int f \psi^c\, \leq \int  \chi_{B_1}\psi_*^{\,c}\,.
		\end{equation}
		
		\noindent
		Summing these inequalities gives~\eqref{toprovereduced} and thus~\eqref{converseineq}. This proves that $\chi_{B_1}$ is a maximiser for $\E(\om_d)$ and then that $B_1$ is a maximiser for $\E_{\textrm{set}}(\om_d)$.
		\medskip
		
		\noindent
		\textit{Step 2. Proof of~\eqref{psistarccstar}.}
		
		As $\psi_*^{\,c}$ and $(\psi^c)^*$ are both continuous radially symmetric functions, to prove~\eqref{psistarccstar} it is sufficient to establish that for any $t> 0$, $\{(\psi^c)^* >t \} \subset \{\psi_*^{\,c} >t \}$, \ie that
		\begin{equation}\label{rearr}
			|\{(\psi^c)^* > t\}| =  |\{\psi^c > t\}| \leq |\{\psi_*^{\,c} > t\}| .
		\end{equation}
		
		\noindent
		Recall that as $\psi \in C_c(\R^d)$ and $k \in C(\R_+, \R_+)$ with $k(x) \to \infty$ as $x \to \infty$, for any $x \in \R^d$ the function $k(|y-x|) - \psi(y)$ admits a minimum on $\R^d$. Thus for any $x \in \R^d$ the infimum defining $\psi^c(x)$ (see Definition~\ref{def_ctransf}) is reached. Recalling that $k$ is also radially symmetric and increasing, we obtain
		\begin{align*}
			\{\psi^c > t\}&=\{x \in \R^d :  \min \{k(|y-x|)-\psi(y) : y \in \R^d\} > t \}\\
			&=\{x\in \R^d :-\psi > t-k(r) \text{ on } \overline{B}_r(x) \ \forall r\geq 0\}\\
			&=\bigcap_{r\geq 0} \{x\in \R^d :-\psi > t-k(r) \text{ on } \overline{B}_r(x)\}\\
			&=\bigcap_{r\geq 0} \{-\psi > t-k(r)\}_r,
		\end{align*}
		
		\noindent
		where for $\Om\subset \R^d$ and $r\geq 0$, $\Om_r$ is defined as $\Om_r:=\{x\in \Om:d(x, \R^d\setminus\Om) > r\}$. In particular,
		\begin{equation}\label{infboundpsic}
			|\{\psi^c> t\} |\le \inf_{r\geq 0}  |\{-\psi > t-k(r)\}_r|.
		\end{equation}
		
		\noindent
		We observe that $\{- \psi > t - k(r)\}$ is an open set for any $t>0$ and $r\geq 0$. We also notice that~\eqref{infboundpsic} holds for all $\psi \in C_c(\R^d)$. In particular, it holds for $\psi_*$. Moreover, as $\psi_*$ is radially non-decreasing by construction, the sets $\{-\psi_*> t-k(r)\}_r$ are open balls centred at the origin and we have in fact
		\begin{equation*}
			|\{\psi_*^{\,c} > t\} |= \inf_{r\geq 0}  |\{-\psi_* > t-k(r)\}_r|.
		\end{equation*}
		
		\smallskip
		
		\noindent
		Let us now prove the following claim.
		\begin{claim*}
			Let $s>0$ and $V> 0$. 
			\begin{enumerate}[(i)]
				\item If $V>\om_d s^d$ then, among open sets $\Om \subset \R^d$ of volume $V$,  $|\Om_s|$ is maximal if and only if $\Om$ is a ball. 
				\item If $V\le\om_d s^d$ then $|\Om_s|=0$ for any set of volume $V$.
			\end{enumerate}
		\end{claim*}
		
		\noindent
		Let $V>0$ and $s>0$ and let $\Om\subset \R^d$ be an open set. We assume without loss of generality that $|\Om| = V$ and $|\Om_s| > 0$. Notice that we always have $\Om_s + B_s \subset \Om$ (but the converse inclusion may fail). By the Brunn-Minkowski inequality (see for instance~\cite{gardnerbrunnmink}) applied to $\Om_s$ and $B_s$, there holds
		\begin{equation}\label{brunnmink}
			V^{1/d}=|\Om|^{1/d}  \geq |\Om_s + B_s|^{1/d} \geq |\Om_s|^{1/d} + |B_s|^{1/d}.
		\end{equation}
		
		\noindent
		If $\Om_s$ is a ball, then $\Om$ is a ball of volume $V$, $\Om = \Om_s + B_s$ and we have equality in~\eqref{brunnmink}. Conversely if we have equality in~\eqref{brunnmink}, by the equality case of the Brunn-Minkowski inequality and the fact that $s>0$, $\Om_s$ is a ball and $|\Om| = |\Om_s +B_s|$, so that $\Om$ is a ball. This proves the first part of the claim.
		
		Regarding the second part, we assume that $|\Om|\le\om_d s^d$ and (by contradiction) that $|\Om_s|>0$. The above reasoning applies and we have $\Om=\Om_s+B_s$ so that $|\Om_s|>0$ implies $|\Om|>|B_s|=\om_d s^d$ and we get a contradiction. This proves the claim.
		
		\medskip
		
		By definition, $\{-\psi > t-k(r)\}$ and $\{-\psi_* > t-k(r)\}$ have the same volume. As a consequence of the claim, for any $t > 0$ and $r > 0$,
		\begin{equation}\label{inegrearr}
			|\{- \psi > t - k(r)\}_r | \leq | \{- \psi_* > t - k(r) \}_r|.
		\end{equation}
		
		\noindent
		Notice that the previous inequality is an equality if $r=0$, as $\Om_0 = \Om$ for any open set $\Om$. Taking the infimum on $r \geq 0$ yields
		\begin{equation}\label{ineqvol}
			|\{\psi^c> t\} |\le \inf_{r \geq 0}  |\{-\psi > t-k(r)\}_r| \leq \inf_{r\geq 0}  |\{-\psi_* > t-k(r)\}_r| = |\{\psi_*^{\,c} > t\} |.
		\end{equation}
		
		\noindent
		This proves~\eqref{rearr} which in turn implies~\eqref{psistarccstar}. 
		\medskip
		
		\begin{center}
			\textit{Part II : Unit balls are the unique maximisers of $\E(\om_d)$.}
		\end{center}
		
		\medskip
		
		\noindent
		\textit{Step 1. Proof of $f=\chi_{\{\psi^c > \psi_*^{\,c}(1)\}}$ (exploiting the equality case in the bathtub principle).}
		
		We now show that any maximiser $f$ is of the form $\chi_{\{\psi^c > \psi_*^{\,c}(1)\}}$. By Lemma~\ref{lem_monotpotball}, $\psi_*^{\,c}$ is radially decreasing on $B_{1}$ and non-increasing on $\R^d$. Thus $\chi_{B_{1}}$ is the only function maximising
		\[
		\sup_{\widetilde{f}} \lt\{\int \widetilde{f} \psi_*^{\,c} \,: 0 \leq \widetilde{f} \leq 1, \, \int \widetilde{f}\,= \om_d \rt\}.
		\]
		
		\noindent
		As $\Ups^*(\chi_{B_{1}}) = \E(\om_d)$, the inequalities in~\eqref{ineqHL}  and~\eqref{lefttoprove} are in fact equalities (and~\eqref{psistarccstar} is also an equality in $B_1$). Namely, there hold 
		\[
		(\psi^c)^* = \psi_*^{\,c} \quad \textrm{ in } \quad B_1,\ \qquad -\int f \psi \,= - \int f^* \psi_*\qquad\text{ and }\qquad \int f \psi^c\, =\int f^* \psi_*^{\,c}\,.
		\]
		This leads to
		\begin{equation*}
			\int f \psi^c \,= \int f^* (\psi^c)^* \,= \int f^* \psi_*^{\,c} \,= \int_{B_{1}}\psi_*^{\,c}\,,
		\end{equation*}
		
		\noindent
		and $f$ is a maximiser of
		\begin{equation*}
			\sup_{\widetilde{f}}\lt\{ \int \widetilde{f} \psi^c\, : 0 \leq \widetilde{f} \leq 1, \, \int \widetilde{f}\,= \om_d\rt\}.
		\end{equation*}
		Let us now prove that $|\{\psi^c>\psi_*^{\,c}(1)\}|=\om_d$ (which with Proposition~\ref{prop_bathtub} yields $f=\chi_{\{\psi^c > \psi_*^{\,c}(1)\}}$). 		 Since $(\psi^c)^*=\psi_*^{\,c}$ in $B_1$, and $\psi_*^{\,c}$ is decreasing in $B_1$ by Lemma~\ref{lem_monotpotball}, there holds for $t\ge \psi^c_*(1)$,
		\begin{equation}\label{equalvol}
			|\{\psi^c >t \}| =|\{(\psi^c)^* >t \}|= |\{\psi_*^{\,c} >t \}|.
		\end{equation}
		Using this for $t=\psi_*^{\,c}(1)$ we get $|\{\psi^c>\psi_*^{\,c}(1)\}|=\om_d$ and we conclude with Proposition~\ref{prop_bathtub} that $f=\chi_{\{\psi^c > \psi_*^{\,c}(1)\}}$.
		
		\medskip
		
		\noindent
		\textit{Step 2. We prove that $\{\psi^c >t \}$ is a ball for $t>\psi^c_*(1)$ (exploiting the equality case in the Brunn-Minkowski inequality).}\\
		\textit{Step 2.a.}
		
		We fix $t> \psi^c_*(1)$.
		Combining~\eqref{equalvol}  and~\eqref{ineqvol}, we get that
		\begin{equation}\label{egrearr}
			|\{\psi^c > t \}| = \inf_{r \geq 0}  |\{-\psi > t-k(r)\}_r| = \inf_{r \geq 0}  |\{-\psi_* > t-k(r)\}_r|= |\{\psi_*^{\,c} > t \}|.
		\end{equation}
		
		\noindent
		The following claim is established in Step 2.b below. 
		\begin{claim*} There exists $r_*=r_*(t)>0$ such that 
			\[
			|\{\psi_*^{\,c} > t \}| = \inf_{r \geq 0}  |\{-\psi_* > t-k(r)\}_r| =  |\{-\psi_*> t-k(r_*)\}_{r_*}|.
			\]
		\end{claim*} 
		Provisionally assuming the claim let us prove that $\{\psi^c >t \}$ is a ball.\\  
		We assume without loss of generality that $|\{\psi_*^{\,c}>t\}|>0$ (otherwise $|\{\psi^c >t \}|\le |\{\psi_*^{\,c} >t \}|=0$ by~\eqref{inegrearr} and the open set $\{\psi^c >t \}$ is empty). Next, the claim,~\eqref{egrearr} and~\eqref{inegrearr} yield that $r_*(t)$ also minimises $\inf_{r \geq 0}  |\{-\psi > t-k(r)\}_r|$.  Thus  by~\eqref{egrearr}, $\{-\psi_* > t-k(r_*(t))\}_{r_*(t)}$ is a ball of positive volume. As $r_*(t)>0$, by the equality case of the claim of Part I, Step 2, the set $\{-\psi > t-k(r_*(t))\}_{r_*(t)}$ is also a ball. As $\{\psi^c > t \} \subset \{-\psi > t-k(r_*(t))\}_{r_*(t)}$, by~\eqref{egrearr} the inclusion is actually an equality. Hence $\{\psi^c >t \}$ is a ball.
		
		\medskip
		
		\noindent
		\textit{Step 2.b. Proof of the claim.}
		
		We first show that there exist $0<r_t<R_t<\infty$ such that
		\begin{equation}\label{reducr}
			\inf_{r \geq 0}  |\{-\psi_* > t-k(r)\}_r| = \inf_{r_t \leq r \leq R_t } |\{-\psi_* > t-k(r)\}_r|.
		\end{equation}
		We start with the upper bound on $r$. By~(H\ref{cont})\&(H\ref{monot}), there exists $R_t$ such that $k(R_t) = t+1$. Hence, if $r > R_t$, $\{-\psi_* >t - k(r) \} = \R^d$. We can thus only  consider the radii $r\le R_t$.
		\medskip		
		
		\noindent
		We now prove the lower bound on $r$. Recall that $t> \psi_*^{\,c}(1)$ and that $\psi_*^{\,c}$ is decreasing in $B_1$. Therefore there exists $R_*(t)<1$ such that
		\[
		\{\psi_*^{\,c}>t\}=B_{R_*(t)}.
		\]
		We set $r_t:= \frac{1-R_*(t)}{2}$ and claim that~\eqref{reducr} holds for this value. To ease notation, let us set  for $r>0$
		\[
		S_r:=\{-\psi_* > t-k(r)\}_r=\{x\in \R^d :-\psi_* > t-k(r) \text{ on } \overline{B}_r(x)\}.
		\]
		We also define $\overline{R}:=\frac{1+R_*(t)}{2}$.
		In order to prove~\eqref{reducr} it is enough to show that
		\begin{equation}\label{psiSr}
			\{\psi_*^{\,c}>t\}=\cap_{r\ge r_t} S_r.
		\end{equation}
		Recalling that the sets $S_r$ are centred balls and that $B_{R_*(t)}\subset B_{\overline{R}}$, we have 
		\[
		\{\psi_*^{\,c}>t\}=\cap_{r\ge 0} (S_r\cap B_{\overline{R}}).
		\]
		We now claim that
		\[
		\cap_{ r\ge r_t} (S_r \cap B_{\overline{R}}) \subset \cap_{r<r_t} (S_r \cap B_{\overline{R}}),
		\]
		which is equivalent to
		\begin{equation}\label{cupcup}
			\cup_{r<r_t} (S_r^c\cap B_{\overline{R}})\subset \cup_{ r\ge r_t} (S^c_r \cap B_{\overline{R}}).
		\end{equation}
		To prove this let $x\in  S_r^c\cap  B_{\overline{R}}$ for some $r<r_t$. By definition of $S_r^c$,
		\[
		\min_{y\in \overline{B}_r(x)} k(r)-\psi_*(y)\le t.
		\]
		In particular since $k$ is increasing, there exists $y\in \overline{B}_r(x)$ such that
		\[
		k(|x-y|)-\psi^*(y)\le t.
		\]
		As $x\in B_{\overline{R}}\subset B_1$, and $(\psi_*^{\,c},\psi_*^{cc})$ are Kantorovitch potentials for the external transport minimising $\Ups(B_1)$ (see Proposition~\ref{prop_Ups_OT}) there exists $z\in B_1^c$ such that $\psi_*^{cc}(z)=\psi_*(z)$ (by~\eqref{psi_KP}) and
		\[
		\psi_*^{\,c}(x)=k(|x-z|)-\psi_*(z)=\min_y k(|x-y|)-\psi_*(y)\le t.
		\]
		Since $z\in B_1^c$ and $x\in B_{\overline{R}}$ we have
		\[r'=|z-x|\ge 1-\overline{R}=\frac{1-R_*(t)}{2}=r_t\]
		and thus
		\[
		\min_{z\in \overline{B}_{r'}(x)} k(r')-\psi_*(z)\le t
		\]
		so that $x\in S_{r'}$.
		This shows~\eqref{cupcup} which implies
		\[
		\{\psi_*^{\,c}>t\}=\cap_{ r\ge r_t} (S_r \cap B_{\overline{R}}).
		\]
		Eventually, we must have $S_r \subset B_{\overline{R}}$ for some $r\ge r_t$  (otherwise $\{\psi_*^{\,c}>t\}=B_{\overline{R}}$ which is absurd). This concludes the proof of~\eqref{psiSr} and thus of~\eqref{reducr}. 
		\medskip
		
		\noindent
		Next, setting 
		\[
		L(r):=|\{-\psi_* > t-k(r)\}_r|,
		\]
		we still have to establish that the infimum of $L$ over $[r_t,R_t]$ is reached. For this we establish that $L$ is lower semi-continuous (together with~\eqref{reducr} this will conclude the proof of the existence of $r_*(t)>0$ minimising $L$ over $\R_+$). We start by noticing that, $r \mapsto |\{-\psi_* > r \}|$ is lower semi-continuous on $\R_+$. Let us denote by $\rho_t(r)$ the radius of the ball $\{-\psi_* > t - k(r) \}$. As $k$ is continuous, the function $r \mapsto \rho_t(r)$ is also lower semi-continuous. Finally, as $L(r) = \om_d[(\rho_t(r) - r)_+]^d$, $L$ is lower semi-continuous as well. This ends the proof of the claim.
		
		\medskip
		
		\noindent
		\textit{Step 3. Conclusion.}
		
		Let now $t_n$ be a decreasing sequence converging to $\psi_*^{\,c}(1)$. We have
		\begin{equation}\label{sequenceballpsi}
			\{\psi^c >\psi_*^{\,c}(1)\} = \bigcup_{n \geq 0} \{\psi^c > t_n \} \qquad \text{and} \qquad \{\psi_*^{\,c} > \psi_*^{\,c}(1)\} = \bigcup_{n \geq 0} \{\psi_*^{\,c} > t_n \} = B_{1}.
		\end{equation}
		
		\noindent
		By~\eqref{egrearr}, for every $n \geq 0$, $\{\psi^c > t_n \} = B_{r_n}(z_n) $, where $r_n$ is the radius of $\{\psi_*^{\,c} > t_n\}$ and $z_n \in \R^d$.  Since $t_n$ is decreasing the sequence $B_{r_n}(z_n)$ is non-decreasing. Moreover, by~\eqref{sequenceballpsi}  $r_n \to 1$ as $n \to \infty$. Hence there exists $z \in \R^d$ such that $\chi_{B_{r_n}} \to \chi_{B_1}(z)$ monotonically in $L^1(\R^d)$ as $n \to \infty$. Eventually~\eqref{sequenceballpsi} implies that $\{\psi^c >\psi_*^{\,c}(1) \} = B_{1}(z)$. Consequently, $f = \chi_{\{\psi^c > \psi_*^{\,c}(1)\}} = \chi_{B_{1}(z)}$. This concludes the proof of the fact that balls are the unique maximisers to~\eqref{MP}.
	\end{proof}

	\section*{Statements and Declarations}
	\noindent\textbf{Data Availability.} There is no data attached to this paper\\
	\textbf{Ethics approval.} We approve the ethics.\\
	\textbf{Conflict of interest.} The authors have no competing interests to declare that are relevant to the content of this article.\\
\textbf{Funding.} No funds, grants, or other support was received.
	\bibliographystyle{acm}
	\bibliography{bib_can_gol_mer.bib}
\end{document}